\documentclass[10pt]{article}
\usepackage{amssymb,amsmath,amsthm,epsfig}
\usepackage{txfonts}
\usepackage{latexsym, enumerate,cite}
\usepackage{eepic}
\usepackage{booktabs}
\usepackage{afterpage}
\usepackage{epic}
\usepackage{graphicx}
\usepackage{color}
\usepackage{ifpdf}
\usepackage{subfigure}
\usepackage{tikz}
\usepackage{dsfont}
\usepackage{multirow}
\usepackage{makecell}
\usepackage{algorithm}
\usepackage{bm}
\usepackage{multirow}
\usepackage{color}
\usepackage[colorlinks, linkcolor=blue,anchorcolor=blue,citecolor=blue,urlcolor=blue]{hyperref}
\usepackage{appendix}
\usepackage{comment}
\graphicspath{{figs/}}

\numberwithin{equation}{section}

\theoremstyle{plain}
\newtheorem{lem}{Lemma}[section]
\newtheorem{thm}[lem]{Theorem}
\newtheorem{cor}{Corollary}[section]
\newtheorem{prop}{Proposition}[section]

\theoremstyle{definition}
\newtheorem{defn}{Definition}[section]

\newtheorem{rem}{Remark}[section]

\DeclareMathOperator{\diag}{diag}

\newcommand{\ddp}[2]{\frac{\partial#1}{\partial#2}}

\newcommand{\p}{\partial}
\newcommand{\ds}{\displaystyle}

\newcommand{\RR}{\mathbb{R}}
\newcommand{\rmr}{\mathrm{r}}

\newcommand{\nm}{\noalign{\smallskip}}

\newcommand{\ZZ}{\mathbb{Z}}
\newcommand{\Scal}{\mathcal{S}}

\newcommand{\Kcal}{\mathcal{K}}

\renewcommand{\(}{\left(}
\renewcommand{\)}{\right)}

\def \e{\ensuremath{\mathrm{e}}}
\def \i{\ensuremath{\mathrm{i}}}
\def \d{\ensuremath{\mathrm{d}}}
\begin{document}

\title{
Mathematical analysis of subwavelength resonant acoustic scattering in multi-layered high-contrast structures 
}

\author{
	{Youjun Deng} \thanks{Corresponding authors. School of Mathematics and Statistics, Central South University, Changsha, 410083, Hunan Province, China. \ \ Email: youjundeng@csu.edu.cn; dengyijun\_001@163.com}
	\and
   {Lingzheng Kong} \thanks{Corresponding authors. School of Mathematics and Statistics, Central South University, Changsha, 410083, Hunan Province, China. Email: math\_klz@csu.edu.cn; math\_klz@163.com}
	\and
	{Yongjian Liu} \thanks{Center for Applied Mathematics of Guangxi, Yulin Normal University, Yulin 537000, P. R. China.  \ \  Email: liuyongjianmaths@126.com}
	\and
	{Liyan Zhu}	\thanks{School of Mathematics and Statistics, Central South University, Changsha, 410083, Hunan Province, China.\ \ Email: math\_zly@csu.edu.cn}
}

\date{}%
\maketitle
\begin{abstract}
Multi-layered structures are widely used in the construction of metamaterial devices to realize various cutting-edge waveguide applications. This paper makes several contributions to the mathematical analysis of subwavelength resonances in a structure of $N$-layer nested resonators. Firstly, based on the Dirichlet-to-Neumann approach, we reduce the solution of the acoustic scattering problem to an $N$-dimensional linear system, and derive the optimal asymptotic characterization of subwavelength resonant frequencies in terms of the eigenvalues of an $N\times N$ tridiagonal matrix, which we refer to as the generalized capacitance matrix. Moreover, we provide a modal decomposition formula for the scattered field, as well as a monopole approximation for the far-field pattern of the acoustic wave scattered by the $N$-layer nested resonators. Finally, some numerical results are presented to corroborate the theoretical findings.

\end{abstract}



\section{Introduction}
In recent years, the subwavelength resonant systems have been extensively studied in the
realms of physics and mathematics and applied to advanced techniques for manipulating wave propagation at subwavelength scales.  In acoustics, classical examples of subwavelength resonant systems include the Minnaert resonance of bubbles in water \cite{Min_1933,AFGLZ_AIHPCAN} and the subsequent development of stable bubble strategies \cite{AFHLY_JDE2019, LLZSIAM2022}. In electromagnetics, such systems encompass plasmonic particles \cite{ACLJFA23,DLZJMPA21,DLbook2024,Ammari2013,FangdengMMA23,KKLSY_JLMS2016,JKMA23,BZ_RMI2019,Ammari2016,KZDF} and high-refractive-index dielectric particles \cite{CGS_JLMS2023,ALLZ_MMS2024}. In linear elasticity, these systems include negative elastic materials \cite{DLL_SIAM2020, DKLZ24, DLL_JST2019, AJKKY_SIAM2017, LL_SIAM2016} and high-contrast elastic materials \cite{ZH_PRB2009, LZScience, LZArxiv, LXarXiv}.  Additionally, contrasting material structures are the effective realization of  various metamaterials with negative material properties through the homogenization theory \cite{AmmariSIMA17, AFLYZQAM19,LWSZ_NM2011}. In particular, the resonant behavior observed at subwavelength scales is a direct consequence of the high contrast in the physical parameters of the medium.

Building on this realization, a class of phononic crystals, consisting of  periodic arrangements of separated subwavelength high-contrast resonators\cite{AFHLY_SIMA2020, AFLYZ_JDE2017, AK_book2018}, has been demonstrated to exhibit bandgaps and is increasingly employed in various waveguide applications \cite{AD_book2024,AD_SIIS2023,AZ_CMP2015}, as well as in imaging as contrast agents \cite{DGS_IPI2021, CCS_IPI2018, SW_SIMA2022}. In most studies on phononic crystals, the structures are typically composed of single-layer (homogeneous) resonators. However, these configurations are often limited by narrow bandgap widths and poor wave filtering performance, making them less suitable for practical engineering applications \cite{PVDD_SCR2010,KMKG_JVA2017,LPDV_PRE2007}. This limitation has spurred the development and investigation of metamaterials with wider bandgaps. In particular, multi-layered high-contrast metamaterials have gained significant attention as promising candidates for subwavelength resonators, owing to their high-tunability and high-quality resonance. Experimental and numerical studies  \cite{LPDV_PRE2007,CMJW_SV2018,KMKG_JVA2017,SPMW_AA2023} have shown that multi-layered concentric radial resonators, in the subwavelength regime, can generate multiple local resonance bandgaps. However, despite considerable progress in both the engineering and physics literature, the mathematical understanding of the origins of subwavelength resonance in multi-layered contrasting media and the mechanism underlying mode splitting remains limited.

In our recent work \cite{DKLLZ_MLHC},  we conducted a rigorous and systematic mathematical analysis of subwavelength resonances in multi-layered high-contrast structures. Specifically, by using layer potential techniques and Gohberg-Sigal theory, we demonstrated that a structure of $N$-layer nested resonators of general shape  exhibits $N$ subwavelength resonant frequencies with positive real parts. Moreover, based on the classical Mie scattering theory, the subwavelength resonances in the $N$-layer concentric radial resonators can be characterized as a $4N\times 4N$ linear system. Through lengthy and tedious calculations, we drived the exact formulas for the subwavelength resonant frequencies for single-resonator, dual-resonator models. For structures with a large number of nested resonators, we provide  numerical computations of resonant modes. Although the $4N\times 4N$ linear system provides a complete characterization of the subwavelength resonances in the $N$-layer concentric radial resonators, it does not directly facilitate the prediction of scattering resonances.  This is primarily due to the computational complexity involved in deriving the asymptotic expansions of the determinant of $4N\times 4N$ block tridiagonal matrix.

The main contribution of this paper is to enhance the mathematical analysis of subwavelength resonant acoustic scattering in multi-layered high-contrast structures. Using the Dirichlet-to-Neumann (DtN) approach previously introduced in \cite{FA_JMPA2024}, we derive an asymptotic characterization of  subwavelength resonant frequencies in terms of the eigenvalues of an $N\times N$ tridiagonal matrix, which we refer to as the generalized capacitance matrix. This formulation provides a highly efficient computational framework for predicting the scattering properties of the multi-layered high-contrast structures.  The capacitance matrix for multiple simply-connected subwavelength resonators  has been extensively studied in \cite{FA_SAM_2022,AD_SIIS2023,AD_book2024,ADY_MMS2020,AFLYZQAM19}. However, to our knowledge, there is no rigorous mathematical analysis of the capacitance matrix for multi-layered doubly-connected subwavelength resonators in the existing literature. This paper highlights the importance of studying multi-layered doubly-connected resonators, as they exhibit unique mathematical features, such as the tridiagonal positive definite property of the capacitance matrix, which are not present in multiple simply-connected subwavelength resonators. Furthermore, we provide a modal decomposition formula for the scattered field, as well as a monopole approximation for the far-field pattern, which demonstrates that the scattered field is greatly enhanced as the impinging frequency approaches one of the $N$ subwavelength resonant frequencies.

The remainder of this paper is organized as follows. In Section \ref{section2}, we first establish the representation formula of the solution of the acoustic scattering problem in multi-layer nested resonators with general shape, and review the existence of subwavelength resonances, the number of which is equal to the number of nested resonators. In Section \ref{section3}, we consider the acoustic scattering  in multi-layered  concentric radial resonators. Firstly, we recall the first characterization of resonances in terms of $4N\times 4N$ block tridiagonal matrix \eqref{MSW} based on spherical wave expansions. Secondly, we give the second characterization of resonances in terms of $2N\times 2N$ block diagonal matrix \eqref{DtNmap} based on Dirichlet-to-Neumann (DtN) map. Finally, based on the DtN approach \cite{FA_JMPA2024}, we reduce the solution of the acoustic scattering problem to that of an $N$-dimensional linear system  \eqref{eq654}, and derive the subwavelength resonances in terms of the eigenvalues of an $N\times N$ tridiagonal matrix \eqref{GEP}, which we refer to as the generalized capacitance matrix.  In Section \ref{MSmonopolar}, we  provide the  modal decomposition and  monopole approximation of multi-layered  concentric radial resonators. In Section \ref{sec5}, numerical computations are presented to to corroborate our theoretical findings. Some concluding remarks are made in Section \ref{sec6}.

\section{Subwavelength resonance in multi-layered high-contrast metamaterials}\label{section2}
\subsection{Problem setting}
We consider the subwavelength resonant phenomenon for the multi-layered  high-contrast (MLHC) metamaterials. Assume that the MLHC metamaterials consist of two types of materials: the host matrix material and the high-contrast material.
We write
\[
D = \cup_{j=1}^{N}D_{j}
\]
to denote the entire resonator-nested, where $D_j$ is the  bounded doubly-connected domain lying between the interior boundary $\Gamma_j^-$ and the exterior boundary $\Gamma_j^+$, $j=1,2,\ldots,N$. Each $\Gamma_{j}^-$ surrounds $\Gamma_{j+1}^+$ and $\Gamma_j^+$ surrounds $\Gamma_j^-$ $(j=1,2,\ldots,N-1)$.
Denote $D_j'$ the region of the gap between the $\Gamma_j^-$ and $\Gamma_{j+1}^+$, $j = 1,2,\ldots,N-1$. Let $D_0'$ and $D_N'$ be the unbounded domain with boundary $\Gamma_{1}^+$ and the bounded domain with boundary $\Gamma_{N}^-$, respectively.  Then, the host matrix material can be written by
\[
\RR^3\setminus{D} = \cup_{j=0}^{N} D'_j.
\]
The configuration of the considered metamaterial is characterized by the density $\rho(x)$ and the bulk modulus $\kappa(x)$ which are given by
 \begin{equation}\label{nestedcomplement}
 \rho(x)
 =\left\{
 \begin{array}{ll}
 \rho_\rmr, & x\in D_j, \; j=1,2,\ldots,N,\\
 \rho, & x\in D_j', \; j=0,1,\ldots,N,
 \end{array}
 \right.\;\; \mbox{ and }\;\;\kappa(x)
 =\left\{
 \begin{array}{ll}
 \kappa_\rmr, &x\in D_j, \; j=1,2,\ldots,N,\\
 \kappa, & x\in D_j', \; j=0,1,\ldots,N.
 \end{array}
 \right.
 \end{equation}
The wave speeds and wavenumbers of resonators and  host matrix materials,respectively, are given by
\begin{equation}\label{auxiliaryparameters}
v_\rmr=\sqrt{\frac{\kappa_\rmr}{\rho_\rmr}}, \quad v=\sqrt{\frac{\kappa}{\rho}}, \quad k_\rmr =\frac{\omega}{v_\rmr}, \quad k=\frac{\omega}{v}.
\end{equation}
We also introduce two dimensionless contrast parameters:
 \begin{equation}\label{contrastparameter}
 \delta=\frac{\rho_\rmr}{\rho}, \,\, \tau= \frac{k_\rmr }{k}= \frac{v}{v_\rmr} =\sqrt{\frac{\rho_\rmr \kappa}{\rho \kappa_\rmr}}.
 \end{equation}
 We will assume that  $
 v = O(1)$, $v_\rmr = O(1)$,  and $\tau = O(1)$; meanwhile
 $\delta \ll 1.$
 This high-contrast assumption is the cause of the underlying system’s subwavelength resonant response and will be at the center of our subsequent analysis.

 We will consider the scattering of a time-harmonic acoustic wave by this MLHC structure. This is described by the Helmholtz system of equations
 \begin{equation} \label{main_equation}
 \begin{cases}
 \ds\Delta u  + \frac{\omega^2}{v^2} u = 0 & \text{in  } \RR^3\setminus D, \\
 \ds\Delta u  + \frac{\omega^2}{v_\rmr^2} u = 0 & \text{in  } D, \\
 \nm
 \ds u|_+ - u|_- = 0 & \text{on } \Gamma^\pm_j, \; j =1,2,\ldots,N,\\
 \nm
 \ds \delta \ddp{u}{\nu}|_+ =  \ddp{u}{\nu}|_-  & \text{on }\Gamma^+_j,\; j =1,2,\ldots,N, \\
 \nm
 \ds  \ddp{u}{\nu}|_+  =\delta \ddp{u}{\nu}|_-  & \text{on }\Gamma^-_j,\; j =1,2,\ldots,N, \\
 \nm
 \ds u^s := u - u^{in} &  \mbox{satisfies the Sommerfeld radiation condition,}
 \end{cases}
 \end{equation}
 where $u^{in}$ is the incoming wave, and  the notation $\nu$ denotes the outward normal on $\Gamma_j^\pm$.
 By the Sommerfeld radiation condition, the scattered wave $u^s$ satisfies
 \begin{equation} \label{eq:src}
 \left(\ddp{}{|x|}-\i  k_0\right)u^s=O(|x|^{-2})\quad \mbox{ as }|x|\to +\infty.
 \end{equation}

 \subsection{Layer potentials}\label{subsec22}
 In this subsection, we briefly introduce the boundary layer potential operators and  then establish the representation formula of the solution \eqref{main_equation}, and we also refer to \cite{AK_book2018,CK_book} for more relevant discussions on layer potential techniques.

 Let $G_k$ be the outgoing fundamental solution to the PDO $ \Delta+k^2$ in $\mathbb{R}^3$, which is given by
 \begin{equation}\label{fundamentalk}
 G_k(x)= - \frac{e^{\i k|x|}}{4 \pi|x|}.
 \end{equation}
 Let $\Omega$ be a bounded domain with a $C^{1,\eta}\, (0<\eta<1)$ boundary $\Gamma$. The single layer potential $\mathcal{S}_{\Gamma}^{k}$ associated with wavenumber $k$ can defined by
 \begin{equation}\label{eq:sh}
 \mathcal{S}_{{\Gamma}}^{k} [\phi](x) =  \int_{{\Gamma}} G_k(x- y) \phi(y) ~\mathrm{d}\sigma(y),  \quad x \in  \mathbb{R}^3,
 \end{equation}
 where $\phi\in L^2(\Gamma)$ is the density function. There hold the following jump relations on the surface \cite{Ammari2007}
 \begin{equation} \label{singlejumpk}
 \frac{\p}{\p\nu}\Scal^k_{\Gamma} [\phi] \Big|_{\pm} = \(\pm \frac{1}{2}I+
 \Kcal_{{\Gamma}}^{k,*}\)[\phi] \quad \mbox{on } {\Gamma},
 \end{equation}
 where the subscripts $+$ and $-$ denote evaluation from outside and inside the boundary $\Gamma$, respectively, and  $\Kcal_{\Gamma}^{k,*}$ is the Neumann-Poincar\'e (NP) operator defined by
 \[
 \Kcal_{\Gamma}^{k,*}[\phi]
 (x) = \mbox{p.v.}\;\int_{{\Gamma}}\frac{\p G_k(x-y)}{\p \nu_x}\phi(y)~\mathrm{d} \sigma(y),\quad x \in  \Gamma.
 \]
 Here p.v. stands for the Cauchy principle value. In what follows, we denote by $\Scal_{\Gamma}$ and $\Kcal_{\Gamma}^{*}$ be the single-layer and Neumann-Poincar\'e operators $\Scal_{\Gamma}^k$ and $\Kcal_{\Gamma}^{k,*}$, by formally taking $k =0$ respectively.

 \subsection{Subwavelength resonance}

 With the help of the layer potentials in subsection \ref{subsec22}, the solution to the Helmholtz system \eqref{main_equation} can be represented by (cf. \cite{DKLLZ_MLHC})
 \begin{equation}\label{Helm_solution}
 u(x) = \begin{cases}
 \ds u^{in} + \mathcal{S}_{\Gamma_1^+}^{k} [\psi_1^+](x), & \quad x \in D_0',\\
 \nm
 \ds\mathcal{S}_{{\Gamma_j^+}}^{k_\rmr} [\phi_j^+](x)+\mathcal{S}_{{\Gamma_{j}^-}}^{k_\rmr} [\psi_{j}^-](x),  & \quad x \in {D_j},\; j = 1,2,\ldots,N,\\
 \nm
 \ds\mathcal{S}_{{\Gamma_j^-}}^{k} [\phi_j^-](x)+\mathcal{S}_{{\Gamma_{j+1}^+}}^{k} [\psi_{j+1}^+](x),  & \quad x \in {D_j'},\; j = 1,2,\ldots,N-1,\\
 \nm
 \ds \mathcal{S}_{{\Gamma_N^-}}^{k} [\phi_N^-](x) ,  & \quad x \in {D'_N},
 \end{cases}
 \end{equation}
 where $\psi_j^\pm,\phi_j^\pm \in  L^2(\Gamma_j^\pm)$, $j =1,2,\ldots,N.$ Using the transmission conditions in \eqref{main_equation}, and the jump relations for the single layer potentials, we can obtain that $\psi_j^\pm$ and $\phi_j^\pm$, $j =1,2,\ldots,N,$  satisfy the following system of boundary integral equations:
\begin{equation}\label{integralsystem}
\mathcal{A}(\omega,\delta)[\Psi]=(u^{in}, \delta \frac{\partial u^{in}}{\partial \nu},0,0,\ldots,0)^T.
\end{equation}
Here
\[
\Psi= (\psi_1^+,\phi_1^+,\psi_1^-,\phi_1^-,\psi_2^+,\phi_2^+,\psi_2^-,\phi_2^-,\ldots,\psi_N^+,\phi_N^+,\psi_N^-,\phi_N^-)^T,
\]
and the $4N$-by-$4N$ matrix type operator $\mathcal{A}(\omega,\delta)$ has the block tridiagonal form
\begin{equation}\label{layeredintegralsystem}
\begin{split}
\mathcal{A}(\omega,\delta):
&=\begin{pmatrix}
 \mathcal{M}_{\Gamma_1^+} & \mathcal{R}_{\Gamma_1^+,\Gamma_1^-} & &  & & \\
 \mathcal{L}_{\Gamma_1^-,\Gamma_1^+} & \mathcal{M}_{\Gamma_1^-} &\mathcal{R}_{\Gamma_1^-,\Gamma_2^+}&  & &\\
 &\mathcal{L}_{\Gamma_2^+,\Gamma_1^-} & \mathcal{M}_{\Gamma_2^+} & \mathcal{R}_{\Gamma_2^+,\Gamma_2^-} & &\\
&  &\mathcal{L}_{\Gamma_2^-,\Gamma_2^+} & \mathcal{M}_{\Gamma_2^-} & \mathcal{R}_{\Gamma_2^-,\Gamma_3^+} & \\
& & & \ddots &\ddots & \ddots& \\
& &  & & \mathcal{L}_{\Gamma_N^+,\Gamma_{N-1}^-} & \mathcal{M}_{\Gamma_N^+} & \mathcal{R}_{\Gamma_N^+,\Gamma_N^-}\\
& &  & &  &\mathcal{L}_{\Gamma_N^-,\Gamma_N^+} & \mathcal{M}_{\Gamma_N^-}
 \end{pmatrix}
 \end{split}
 \end{equation}
  where $\mathcal{M}_{\Gamma_i^\pm}$ is the self-interaction of the exterior  and the interior boundaries for the $i$-th resonator, respectively,  defined by
  \begin{equation}\label{contrastselfinteract}
  \quad \mathcal{M}_{\Gamma_i^+}=
  \begin{pmatrix}
  -\mathcal{S}^{k}_{\Gamma_i^+} & \mathcal{S}^{k_\rmr }_{\Gamma_i^+} \\
  -\delta( \frac{1}{2}I+ \mathcal{K}_{\Gamma_i^+}^{k, *}) & -\frac{1}{2}I+ \mathcal{K}_{\Gamma_i^+}^{k_\rmr , *}
  \end{pmatrix}, \quad
  \mathcal{M}_{\Gamma_i^-}=\begin{pmatrix}
  -\mathcal{S}^{k_\rmr }_{\Gamma_{i}^-} & \mathcal{S}^{k}_{\Gamma_{i}^-} \\
  -( \frac{1}{2}I+ \mathcal{K}_{\Gamma_{i}^-}^{k_\rmr , *}) & \delta(-\frac{1}{2}I+ \mathcal{K}_{\Gamma_{i}^-}^{k, *} )
  \end{pmatrix},
  \end{equation}
$\mathcal{L}_{\Gamma_i^-,\Gamma_i^+}$  and $\mathcal{R}_{\Gamma_i^-,\Gamma_{i+1}^+}$ encode the effect of the exterior boundaries of the $i$-th resonator and the $(i+1)$-th resonator, respectively, on the interior boundary of the $i$-th resonator, as defined by
\begin{equation}\label{contrastinteractEIi}
\mathcal{L}_{\Gamma_i^-,\Gamma_i^+} = \begin{pmatrix}
\ds 0&-\mathcal{S}^{k_\rmr }_{\Gamma_{i}^-,\Gamma_{i}^+}\\
\nm
\ds 0&-\mathcal{K}_{\Gamma_{i}^-,\Gamma_{i}^+}^{k_\rmr , *}
\end{pmatrix}, \quad
\mathcal{R}_{\Gamma_i^-,\Gamma_{i+1}^+}=
\begin{pmatrix}
\ds\mathcal{S}^{k}_{\Gamma_{i}^-,\Gamma_{i+1}^+}& 0\\
\nm
\ds \delta\mathcal{K}_{\Gamma_{i}^-,\Gamma_{i+1}^+}^{k, *}& 0
\end{pmatrix},
\end{equation}
$\mathcal{L}_{\Gamma_{i}^+,\Gamma_{i-1}^-}$  and $\mathcal{R}_{\Gamma_{i}^+,\Gamma_{i}^-}$ encode the effect of the interior boundaries of the $(i-1)$-th resonator and the $i$-th resonator, respectively, on the exterior boundary of the $i$-th resonator, as defined by
\begin{equation}\label{contrastinteractL}
\mathcal{L}_{\Gamma_{i}^+,\Gamma_{i-1}^-} = \begin{pmatrix}
0&-\mathcal{S}^{k}_{\Gamma_{i}^+,\Gamma_{i-1}^-}\\
\nm
0&-\delta \mathcal{K}_{\Gamma_{i}^+,\Gamma_{i-1}^-}^{k, *}
\end{pmatrix}, \quad
\mathcal{R}_{{\Gamma_{i}^+,\Gamma_{i}^-}}=
\begin{pmatrix}
\ds \mathcal{S}^{k_\rmr }_{\Gamma_{i}^+,\Gamma_{i}^-}& 0\\
\nm
\ds \mathcal{K}_{\Gamma_{i}^+,\Gamma_{i}^-}^{k_\rmr , *}& 0
\end{pmatrix}
.
\end{equation}
Here, we introduced the operators $\mathcal{S}^{k}_{\Gamma_{i},\Gamma_{j}}:L^2(\Gamma_{j})\to L^2(\Gamma_i)$ and $\mathcal{K}_{\Gamma_{i},\Gamma_{j}}^{k, *}:L^2(\Gamma_{j})\to L^2(\Gamma_i)$ as defined respectively by
\[
\mathcal{S}^{k}_{\Gamma_{i},\Gamma_{j}}[\varphi] = \mathcal{S}^{k}_{\Gamma_{j}}[\varphi] \big|_{\Gamma_i}\;\mbox{ and }\; \mathcal{K}_{\Gamma_{i},\Gamma_{j}}^{k, *}[\varphi] = \ddp{}{\nu_i}\mathcal{S}^{k}_{\Gamma_{j}}[\varphi] \big|_{\Gamma_i}, \mbox{ for } \forall \varphi\in L^2(\Gamma_{j}).
\]

In order to understand the resonant behavior of the $N$-layer nested scatterers, we shall give the definition of the subwavelength resonant frequencies and resonant modes of the system  based on the high contrast $\delta$  between the materials.

\begin{defn}\label{defn:resonance}
	Given $\delta>0$, a subwavelength resonant frequency (eigenfrequency) $\omega=\omega(\delta)\in\mathbb{C}$ is defined to be such that

	$\,$(i) in the case that $u^{in} = 0$, there exists a nontrivial solution to \eqref{main_equation}, known as an associated resonant mode (eigenmode);
	
	(ii) $\omega$ depends continuously on $\delta$ and satisfies $\omega\to0$ as $\delta\to0$.
\end{defn}

When the material parameters are real, it is easy to see that $\overline{\mathcal{A}(\omega,\delta)}=\mathcal{A}(-\overline{\omega},\delta)$, from which we can see that the subwavelength resonant frequencies will be symmetric about the imaginary axis.

\begin{lem}[see \cite{dyatlov2019mathematical}] \label{symmetric_res}
	The set of subwavelength resonant frequencies is symmetric about the imaginary axis. That is, if $\omega$ is such that $\mathcal{A}(\omega,\delta)[\Psi] = 0$ holds true for some nontrivical $\Psi\in \mathcal{H}$, then it will also hold that
	\begin{equation*}
	\mathcal{A}(-\overline{\omega},\delta)
	\left[\overline{\Psi}\right]
	= 0.
	\end{equation*}
\end{lem}
Based on Lemma \ref{symmetric_res}, we will subsequently state results only for the subwavelength resonant frequencies with positive real parts.
The existence of subwavelength resonant frequencies is given by the following Theorem, which was proved in \cite{DKLLZ_MLHC} by using the generalized Rouch\'e theorem \cite{AK_book2018}.

\begin{thm}[\cite{DKLLZ_MLHC}]\label{thmres}
	Consider a structure of $N$-layer  subwavelength nested resonators in $\mathbb{R}^3$.
	For sufficiently small $\delta>0$, there exist $N$ subwavelength resonant frequencies $\omega^+_{1}(\delta),\omega^+_{2}(\delta),\dots,\omega^+_{N}(\delta)$ with positive real parts.
\end{thm}

\section{Subwavelength resonance in multi-layered concentric balls}\label{section3}

It is known that the subwavelength resonant frequency is associated with the shape of the resonators. However, it has been observed that breaking the rotational symmetry of the resonators does not induce mode splitting. In other words, altering the shape of the resonator does not affect the number of subwavelength resonant frequencies. Based on this observation and from a construction perspective, we shall consider  the Helmholtz system \eqref{main_equation} where $D$ is a multi-layered concentric ball, as illustrated in Figure \ref{MLHCCB}. Precisely, we give a sequence of resonators, $D_j,$ $j=1,2,\ldots,N$, by
\begin{equation}
D_j = \{r_{j}^-<r\leqslant r_{j}^+\},
\end{equation}
 the host matrix material $D_j'$, $j=0,1,\ldots,N,$ by
\begin{equation}\label{eq:aj}
D_{0}':=\{r>r_{1}^+\}, \quad D_j':=\{r_{j+1}^+<r\leqslant r_{j}^-\}, \quad  j=1,2,\ldots, N-1, \quad D_N':=\{r\leqslant r_N^-\},
\end{equation}
and the interfaces between the adjacent layers can be rewritten by
\begin{equation}\label{interface}
\Gamma_j^\pm:=\left\{|x|=r_j^\pm\right\}, \quad  j=1,2,\ldots,N,
\end{equation}
where $N\in \mathbb{N}$ and $r_j^\pm>0$.

\begin{figure}[htbp]
	\centering
	\includegraphics[scale=0.33]{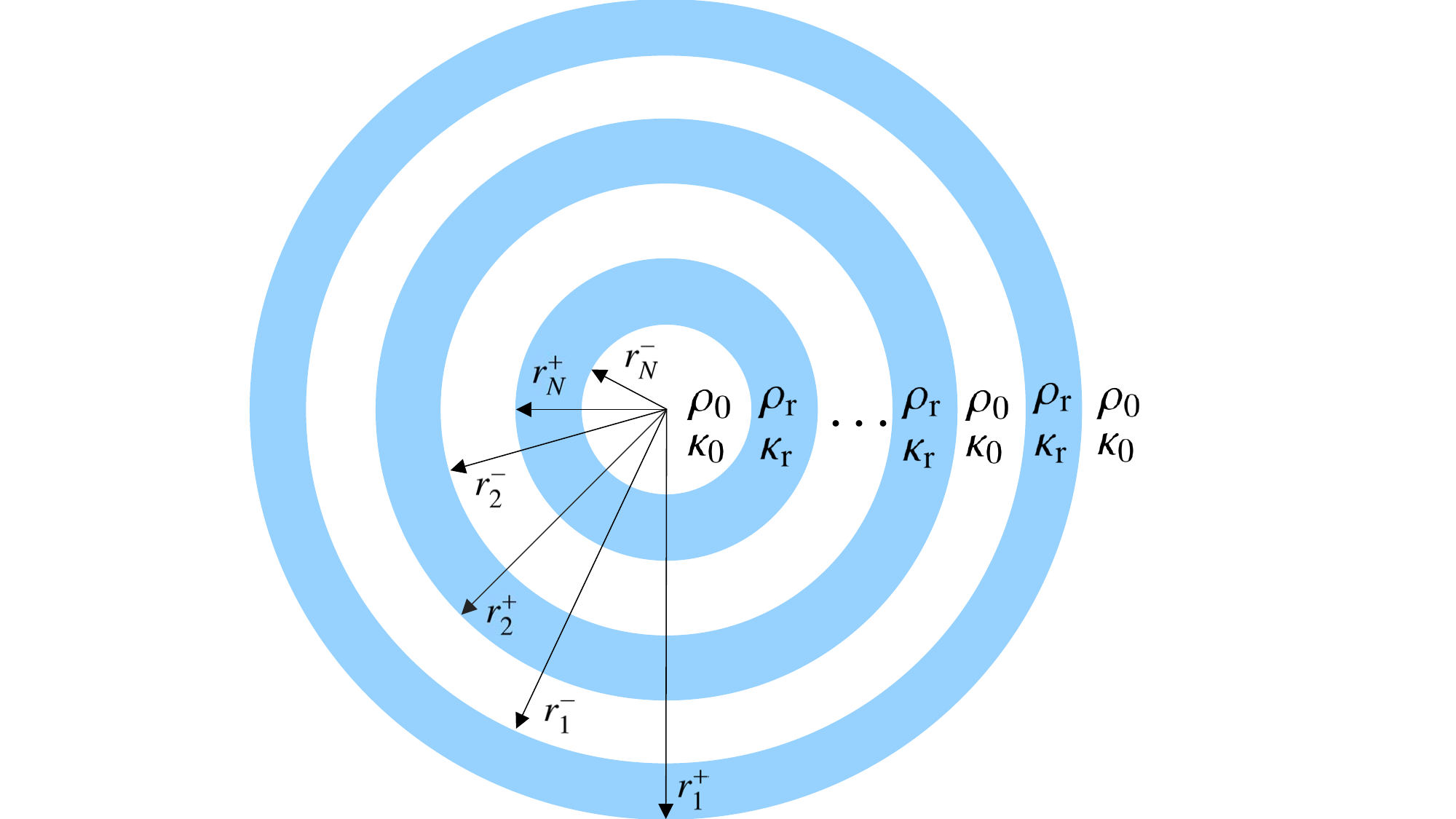}
	\caption{ Schematic illustration of a structure of $N$-layer nested  resonators.}\label{MLHCCB}
\end{figure}

\subsection{A first characterization of resonances based on spherical wave expansions}\label{SWE}
In this subsection, we recall the first characterization of resonances based on spherical wave expansions. By rather lengthy and tedious calculations, the exact formulas for the eigenfrequencies of single-resonator and dual-resonator, as well as numerical computations for any finite multi-layer nested resonators, have been derived and presented  in \cite{DKLLZ_MLHC}.

Let $j_n(t)$ and $h^{(1)}_n(t)$ be the spherical Bessel and Hankel functions of the
first kind of order $n$, respectively, and $Y_n(\hat{x})$ be the spherical harmonics.
By using spherical coordinates,  the solution $u$ to \eqref{main_equation}, with the material parameters given in \eqref{nestedcomplement}--\eqref{contrastparameter}, can be written as
\begin{equation}\label{Helm_solution2}
u(x) = \begin{cases}
\ds u^{in}+\sum_{n=0}^{+\infty}a^+_{1,n}h^{(1)}_n(kr)Y_n, & \quad x \in D_0',\\
\nm
\ds \sum_{n=0}^{+\infty}\(b^+_{j,n}j_n(k_\rmr  r)Y_n + a^-_{j,n}h^{(1)}_n(k_\rmr  r)Y_n\),  & \quad x \in {D_j},\; j=1,2,\ldots,N,\\
\nm
\ds \sum_{n=0}^{+\infty}\(b^-_{j,n}j_n(k r)Y_n + a^+_{j+1,n}h^{(1)}_n(k r)Y_n\),  & \quad x \in {D_j'},\; j=1,2,\ldots,N,
\end{cases}
\end{equation}
where $a^+_{N+1,n}=0$. By using the transmission conditions across $\Gamma^\pm_j$, $j=1,2,\ldots,N$,  we find that the constants satisfy
\[
\bm{A}_{(n)}(\omega,\delta)
[\bm{a}^\pm] = (u^{in}|_{\Gamma_1^+},\delta \frac{\partial u^{in}}{\partial \nu}|_{\Gamma_1^+}, 0,0, \cdots,0)^T,
\]
for all $n\in \mathbb{N}\cup\{0\}$. Here
\[
\bm{a}^\pm = \(a^+_{1,n},b^+_{1,n}, a^-_{1,n},b^-_{1,n}, \cdots,a^+_{N,n},b^+_{N,n},a^-_{N,n},b^-_{N,n}\)^T
\]
and the $4N$-by-$4N$ matrix type operator $\bm{A}_{(n)}(\omega,\delta)$ has the block tridiagonal form
\begin{equation}\label{4Nby4N}
\bm{A}_{(n)}(\omega,\delta):=\begin{pmatrix}
{M}^+_{1,n} & {R}^{+,-}_{1,1,n} & &  & && \\
\nm
{L}^{-,+}_{1,1,n} & {M}^-_{1,n} &{R}^{-,+}_{1,2,n}& &&&\\
\nm
&{L}^{+,-}_{2,1,n} & {M}^+_{2,n} & {R}^{+,-}_{2,2,n} & &&\\
\nm
&&{L}^{-,+}_{2,2,n} & {M}^-_{2,n} & {R}^{-,+}_{2,3,n} & &&\\
&& & \ddots &\ddots & \ddots& \\
&&  & & {L}^{+,-}_{N,N-1,n} & {M}^+_{N,n} & {R}^{+,-}_{N,N,n}\\
\nm
&&  & &  &{L}^{-,+}_{N,N,n} & {M}^-_{N,n}
\end{pmatrix},
\end{equation}
where
\begin{equation}\label{contrastselfinteract2}
 {M}_{i,n}^+ = \begin{pmatrix}
 -h^{(1)}_n(kr_i^+) & j_n(k_\rmr  r_i^+) \\
 \nm
 -\delta h^{(1)'}_n(kr_i^+) & \tau j'_n(k_\rmr  r_i^+)
 \end{pmatrix},
 \quad
 {M}_{i,n}^- = \begin{pmatrix}
 -h^{(1)}_n(k_\rmr  r_i^-) & j_n(k r_i^-) \\
 \nm
 -\tau  h^{(1)'}_n(k_\rmr  r_i^-) & \delta j'_n(k r_i^-)
 \end{pmatrix},
\end{equation}
\[
{L}^{-,+}_{i,i,n} =  \begin{pmatrix}
0&-j_n(k_\rmr r_i^-)\\
\nm
0&-\tau j'_n(k_\rmr r_i^-)
\end{pmatrix},
\quad
{R}^{+,-}_{i,i,n} = \begin{pmatrix}
h^{(1)}_n(k_\rmr  r_i^+)& 0\\
\nm
\tau h^{(1)'}_n(k_\rmr  r_i^+)& 0
\end{pmatrix},
\]
and
\[
{L}^{+,-}_{i,i-1,n} = \begin{pmatrix}
0&-j_n(kr_i^+)\\
\nm
0&-\delta j'_n(kr_i^+)
\end{pmatrix},
\quad
{R}^{-,+}_{i,i+1,n} = \begin{pmatrix}
h^{(1)}_n(k r_i^-)& 0\\
\nm
\delta h^{(1)'}_n(k r_i^-)& 0
\end{pmatrix}.
\]

It is important to emphasize that, from a physical perspective, we are concerned with the resonance of nested materials, which corresponds to the system's lowest resonant frequency. At this frequency, the system exhibits a factor corresponding to $n = 0$, as the lowest resonance is characterized by the fewest number of oscillations \cite{DKLLZ_MLHC,AFHLY_JDE2019}.
It can be seen in the proofs of \cite[Theorems 4.1-4.2]{DKLLZ_MLHC}  that the primary reason for mode splitting lies in the fact that as the number of nested resonators increases, the degree of the corresponding characteristic polynomial also increases, while the type of resonance (which consists solely of monopolar resonances) remains unchanged.  The mathematical justification for the monopole approximation is given in Section \ref{MSmonopolar}.
Consequently,  the $4N$-by-$4N$ matrix
\begin{equation}\label{MSW}
\bm{A}(\omega,\delta):=\bm{A}_{(0)}(\omega,\delta)
\end{equation}
becomes singular.

\begin{prop}\label{FCR}
Subwavelength resonant frequencies $\omega=\omega(\delta)\in\mathbb{C}$ are determined by $\det(\bm{A}(\omega(\delta),\delta))=0$.
\end{prop}

Although Proposition \ref{FCR} provides a complete characterization of the solution to \eqref{main_equation}, it does not directly facilitate the prediction of scattering resonances.
This is primarily due to the computational complexity involved in deriving the asymptotic expansions of the determinant $\det(\bm{A}(\omega,\delta))$ of the $4N$-by-$4N$ matrix $\bm{A}(\omega,\delta)$ by using the following asymptotic expansions:
\begin{equation}\label{besssmall}
j_0(t)=\frac{\sin t}{t} =  1 - \frac{t^2}{6} + \frac{t^4}{120} + O(t^6),
\end{equation}
\begin{equation}\label{hankelsmall}
h^{(1)}_0(t) =\frac{\sin t}{t}-\i \frac{\cos t}{t} = 1 - \frac{t^2}{6} + \frac{t^4}{120} + \mathrm{i}(-\frac{1}{t} + \frac{t}{2} - \frac{t^3}{24})+ O(t^5),
\end{equation}
for $t\ll 1$.
We remark in Subsection \ref{D2Nmap} that the scattering problem \eqref{main_equation} can be reformulated as a $2N\times 2N$ linear system based on the Dirichlet-to-Neumann (DtN) map, which gives a second characterization of the resonances.

\subsection{ A second characterization of resonances based on the DtN map}\label{D2Nmap}

In this section, we characterize the DtN map of the Helmholtz problem \eqref{main_equation} in MLHC metamaterial $D = \cup_{j=1}^{N}D_{j}$, with the aim of applying the DtN approach developed in \cite{FA_JMPA2024} to analyze subwavelength resonances.

The following lemma provides an explicit expression for the solution
to exterior Dirichlet  problems on $\RR^3\setminus D$.

\begin{lem}
	\label{def:DTN}
	For any $k\in\mathbb{C}\backslash\{ m\pi/(r_j^--r_{j+1}^+)\,|\, m\in\ZZ\backslash\{0\},\, 1\leq j\leq N-1 \}$ and any vector $  (f_j^{\pm})_{1\leq j\leq N} \in\mathbb{C}^{2N}$, there exists a unique solution  $v_f\in H^{1}(\RR^3)$ to the exterior Dirichlet problem:
	\begin{equation} \label{exterior_problem}
	\begin{cases}
	\ds\Delta v_f + k^2 v_f = 0 & \text{in  } \RR^3\setminus D, \\
	\nm
	\ds  {v_f}|_{\Gamma_j^\pm}   = f_j^\pm &  \text{for } j =1,2,\ldots,N, \\
	\nm
	\ds v_f &  \mbox{satisfies the Sommerfeld radiation condition.}
	\end{cases}
	\end{equation}
	Furthermore, when $k\neq 0$, the solution $v_f$ can be represented by
	\begin{equation}
	\label{ext_prob_solu}
	v_f = \begin{cases}
	\ds f_1^+\frac{h_0^{(1)}(kr)}{h_0^{(1)}(k r_1^+)}, & x\in D_0', \\
	\nm
	\ds  a_j j_0(kr)+b_j h_0^{(1)}(kr), & x\in D_j'\quad j =1,2,\ldots,N-1, \\
	\nm
	\ds f_N^-\frac{j_0(kr)}{j_0(k r_N^-)}, & x\in D_N', \\
	\end{cases}
	\end{equation}
	where $a_j$ and $b_j$ are given by
	\begin{equation}
	\label{ajbj}
	\begin{pmatrix}
	a_j\\
	b_j
	\end{pmatrix} = \frac{\i k^2 r_j^- r_{j+1}^+}{\sin(k(r_j^--r_{j+1}^+))} \begin{pmatrix}
h_0^{(1)}(kr_{j+1}^+) & -h_0^{(1)}(kr_j^-) \\
	-j_0(kr_{j+1}^+) & 	j_0(kr_j^-)
	\end{pmatrix}
	\begin{pmatrix}
	f_j^{-}\\
	f_{j+1}^{+}
	\end{pmatrix}.
	\end{equation}
\end{lem}
\begin{proof}[\bf Proof]
	When $k\neq 0$, the solution $v_f$ to \eqref{exterior_problem} may be represented as \eqref{ext_prob_solu}. It follows from the Dirichlet boundary conditions of \eqref{exterior_problem} that the constants $a_j$ and $b_j $, $j=1,2,\ldots,N-1$, satisfy
	\[
	\begin{pmatrix}
	j_0(kr_j^-) & h_0^{(1)}(kr_j^-) \\
	j_0(kr_{j+1}^+) & 	h_0^{(1)}(kr_{j+1}^+)
	\end{pmatrix}
	\begin{pmatrix}
	a_j\\
	b_j
	\end{pmatrix} = \begin{pmatrix}
	f_j^{-}\\
	f_{j+1}^{+}
	\end{pmatrix}.
	\]
	Inverting above equality and using the fact
	\[
	j_0(kr_j^-) h_0^{(1)}(kr_{j+1}^+)-j_0(kr_{j+1}^+))h_0^{(1)}(kr_j^-)= \frac{\sin(k(r_j^--r_{j+1}^+))}{\i k^2 r_j^- r_{j+1}^+},
	\]
	we can get \eqref{ajbj}.
When $k =0$, the uniqueness of exterior Dirichlet problem follows from \cite[Proposition 3.1]{GBF_book1995}. The proof is complete.
\end{proof}

\begin{defn}
	For any $k\in\mathbb{C}\backslash\{ m\pi/(r_j^--r_{j+1}^+)\,|\, m\in\ZZ\backslash\{0\},\, 1\leq j\leq N-1 \}$, the Dirichlet-to-Neumann (DtN) map with wave number $k$ is
	the operator $\mathcal{T}^k : \mathbb{C}^{2N} \rightarrow \mathbb{C}^{2N}$ defined by
	\begin{equation}
	\label{DtN_map}
	\mathcal{T}^{k}[(f_j^{\pm})_{1\leq j\leq N}]=
	\left(  \frac{\partial v_f}{\partial v}|_{\Gamma_j^\pm}\right)_{1 \leq j \leq N},
	\end{equation}
where $v_f$ is the unique solution to \eqref{exterior_problem}.
\end{defn}

\begin{rem}
	The condition that  $k\in\mathbb{C}\backslash\{ m\pi/(r_j^--r_{j+1}^+)\,|\, m\in\ZZ\backslash\{0\},\, 1\leq j\leq N-1 \}$ is equivalent to the assumption that $k^2$ is not a Dirichlet eigenvalue of $-\Delta$ on $\RR^3 \setminus D$, which guarantees the uniqueness of exterior Dirichlet problem \eqref{exterior_problem} and then the well-defined of DtN map \eqref{D2Nmap}.
\end{rem}

We give the exact matrix representation of the DtN map $\mathcal{T}^{k}$ in the following proposition.

\begin{prop} \label{prop:DTN}
	For any $k\in\mathbb{C}\backslash\{ n\pi/(r_j^--r_{j+1}^+)\,|\, n\in\ZZ\backslash\{0\},\, 1\leq j\leq N-1 \}$,  $f:= (f_j^{\pm})_{1\leq j\leq N}$, the DtN map $\mathcal{T}^{k}[f]:= (\mathcal{T}^{k}[ f]_{j}^{\pm})_{1\leq j\leq N}$ admits the following exact matrix representation: 
	\begin{equation}\label{DtN_MR}
	\begin{pmatrix}
	\mathcal{T}^{k}[f]_1^{+}\\
	\mathcal{T}^{k}[f]_{1}^{-}\\
	\mathcal{T}^{k}[f]_2^{+}\\
	\vdots\\
	\mathcal{T}^{k}[f]_{N-1}^{-}\\
	\mathcal{T}^{k}[f]_N^{+}\\
	\mathcal{T}^{k}[f]_{N}^{-}
	\end{pmatrix} =  \begin{pmatrix}
	-\frac{1}{r_1^+}+\i k&&&&&\\
	&A^k_{1,2}&&&&\\
	&&A^k_{2,3}&&&\\
	&&&\ddots &&\\
	&&&&A^k_{N-1,N}&\\
	&&&&&\frac{kr_N^-\cos(kr_N^-)-\sin(kr_N^-)}{r_N^-\sin (kr_N^-)}
	\end{pmatrix} \begin{pmatrix}
	f_1^{+}\\
	f_{1}^{-}\\
	f_2^{+}\\
	\vdots\\
	f_{N-1}^{-}\\
	f_N^{+}\\
	f_{N}^{-}
	\end{pmatrix},
	\end{equation}
	where for $j=1,2,\ldots,N-1,$
\begin{equation}\label{Akjjp1}
A^k_{j,j+1} :=
\begin{pmatrix}
\ds\frac{kr_j^-\cos(k(r_j^--r_{j+1}^+))-\sin (k(r_j^--r_{j+1}^+))}{r_j^- \sin(k(r_j^--r_{j+1}^+))}&  \ds -\frac{kr_{j+1}^+}{r_j^- \sin(k(r_j^--r_{j+1}^+))}\\
\nm
\ds \frac{kr_j^-}{r_{j+1}^+ \sin(k(r_j^--r_{j+1}^+))}&  \ds - \frac{kr_{j+1}^+\cos(k(r_j^--r_{j+1}^+))+\sin (k(r_j^--r_{j+1}^+))}{ r_{j+1}^+ \sin(k(r_j^--r_{j+1}^+))}
\end{pmatrix}.
\end{equation}
\end{prop}

\begin{proof}[\bf Proof]
In view of \eqref{ext_prob_solu}--\eqref{DtN_map}, we have that for $j=1,2,\ldots,N-1$,
\[
\begin{aligned}
\begin{pmatrix}
\mathcal{T}^{k}[f]_j^{-}\\
\mathcal{T}^{k}[f]_{j+1}^{+}
\end{pmatrix} &= k\begin{pmatrix}
 j'_0(kr_j^-)&  h_0^{(1)'}(kr_j^-)\\
\nm
j'_0(kr_{j+1}^+)&  h_0^{(1)'}(kr_{j+1}^+)
\end{pmatrix}
\begin{pmatrix}
a_j\\
b_j
\end{pmatrix}\\
& =\frac{\i k^3 r_j^- r_{j+1}^+}{\sin(k(r_j^--r_{j+1}^+))}\begin{pmatrix}
j'_0(kr_j^-)&  h_0^{(1)'}(kr_j^-)\\
\nm
j'_0(kr_{j+1}^+)&  h_0^{(1)'}(kr_{j+1}^+)
\end{pmatrix}
\begin{pmatrix}
h_0^{(1)}(kr_{j+1}^+) & -h_0^{(1)}(kr_j^-) \\
-j_0(kr_{j+1}^+) & 	j_0(kr_j^-)
\end{pmatrix}
\begin{pmatrix}
f_j^{-}\\
f_{j+1}^{+}
\end{pmatrix}\\
& = A^k_{j,j+1}
\begin{pmatrix}
f_j^{-}\\
f_{j+1}^{+}
\end{pmatrix}.
\end{aligned}
\]
Moreover, we have
\[
\mathcal{T}^{k}[f]_{1}^{+} = \frac{kh_0^{(1)'}(kr_1^+)}{h_0^{(1)}(k r_1^+)} f_1^+ = \(-\frac{1}{r_1^+}+\i k\)f_1^+,
\]
 and
\[
\mathcal{T}^{k}[f]_N^{-} = \frac{kj_0'(kr_N^-)}{j_0(k r_N^-)} f_N^-  =\frac{kr_N^-\cos(kr_N^-)-\sin(kr_N^-)}{r_N^-\sin (kr_N^-)}f_N^-.
\]
The proof is complete.
\end{proof}

It can be verified that the solution $v_f$ to \eqref{exterior_problem} with $k\neq 0$ converges as $k \rightarrow 0$ to the solution to \eqref{exterior_problem} with $k=0$. As expected from the matrix representation \eqref{DtN_MR}, the operator $\mathcal{T}^{k}$  is analytic in a neighborhood of $k=0$. In all what follows, we denote by $R$ the convergence radius
\[
R:=\frac{\pi }{\max_{1\le j\le N-1}\{r_j^--r_{j+1}^+\}}.
\]
\begin{cor}\label{coranalytic_DtN}
	The DtN map $\mathcal{T}^{k}$ can be extended to a holomorphic $2N\times 2N$ matrix with respect to the wave number $k\in  \{k\in \mathbb{C}:| k| <R\}$. Therefore, there exists a family of $2N\times 2N$ matrices $(\mathcal{T}_{n})_{n\in \mathbb{N}}$ such that the following convergent series representation holds:
	\begin{equation}
	\label{DTN_expansion}
	\mathcal{T}^{k} =
	\sum_{n=0}^{+ \infty} k^{n} \mathcal{T}_{n},\qquad \forall k\in\mathbb{C}
	\text{ with }|k|<R.
	\end{equation}
Particularly,  the matrices $\mathcal{T}_{0}$ and $\mathcal{T}_{1}$ can be, respectively, given by
\begin{equation}\label{T_0}
\mathcal{T}_{0}= \begin{pmatrix}
-\frac{1}{r_1^+}&&&&&\\
&A^0_{1,2}&&&&\\
&&A^0_{2,3}&&&\\
&&&\ddots &&\\
&&&&A^0_{N-1,N}&\\
&&&&&0
\end{pmatrix}
\text{ and }
\mathcal{T}_{1} =
\begin{pmatrix}
\i&&&&&\\
&0&&&&\\
&&0&&&\\
&&&\ddots &&\\
&&&&0&\\
&&&&&0
\end{pmatrix},
\end{equation}
where
\[
A^0_{j,j+1} =
\begin{pmatrix}
\ds \frac{r_{j+1}^+}{r_j^- (r_j^--r_{j+1}^+)}&  \ds -\frac{r_{j+1}^+}{r_j^- (r_j^--r_{j+1}^+)}\\
\nm
\ds \frac{r_j^-}{r_{j+1}^+ (r_j^--r_{j+1}^+)}&  \ds -\frac{r_{j}^-}{r_{j+1}^+ (r_j^--r_{j+1}^+)}
\end{pmatrix},\quad j =1,2,\ldots,N-1.
\]
\end{cor}
\begin{proof}[\bf Proof]
The result is immediate by noticing that the matrix
$A^k_{j,j+1}$ of \eqref{Akjjp1} is analytic with respect to the $k$ satisfying $|k|(r_j^--r_{j+1}^+)<\pi$ for $j=1,2,\ldots,N-1$. From \eqref{besssmall} and \eqref{hankelsmall}, we have
\begin{equation}
\frac{kj_0'(k r_N^-)}{j_0(k r_N^-)} = O(k^2)\;\text{ and } \,
\frac{kh_0^{(1)'}(kr_1^+)}{h_0^{(1)}(k r_1^+)} = -\frac{1}{r_1^+} +\i k.
\end{equation}
This, together with $A^k_{j,j+1} = A^0_{j,j+1}+O(k^2)$, implies that \eqref{T_0} holds. The proof is complete.
\end{proof}

\begin{rem}
	Based on \eqref{DtN_MR} and \eqref{T_0}, we have that for a vector $f:= (f^\pm_i )_{1\leq i\leq N}\in \mathbb{C}^{2N}$,
	\begin{equation} \label{T_0expre}
	\begin{cases}
	\ds\mathcal{T}_0[f]_1^{+} = -\frac{1}{r_1^+}f_1^{+}, &  \\
	\nm
	\ds   \mathcal{T}_0[f]_j^{+} = \frac{r_{j-1}^-}{r_{j}^+ (r_{j-1}^--r_{j}^+)}  f_{j-1}^- -\frac{r_{j-1}^-}{r_{j}^+ (r_{j-1}^--r_{j}^+)} f_{j}^+, &j =2,3,\ldots,N, \\
	\nm
	\ds   \mathcal{T}_0[f]_j^{-} = \frac{r_{j+1}^+}{r_j^- (r_j^--r_{j+1}^+)} f_{j}^-  -\frac{r_{j+1}^+}{r_j^- (r_j^--r_{j+1}^+)}f_{j+1}^+, & j =1,2,\ldots,N-1,  \\
	\nm
	\ds  \mathcal{T}_0[f]_N^{-} = 0.&
	\end{cases}
	\end{equation}
\end{rem}

Based upon the properties of the DtN map above, we now provide a second characterization of resonances in the MLHC structure $D = \cup_{j=1}^{N}D_{j}$ as illustrated in Figure \ref{MLHCCB}.
The solution to the acoustic scattering problem \eqref{main_equation} in $D$ can be rewritten in terms of the DtN map $\mathcal{T}^k$ for $u\in H^{1}(D)$
\begin{equation} \label{D2N}
\begin{cases}
\ds\Delta u  + \frac{\omega^2}{v_\rmr^2} u = 0 & \text{in  } D, \\
\nm
\ds   \ddp{u}{\nu}|_\mp = \delta \mathcal{T}^k[u-u^{in}]_j^\pm + \delta\ddp{u^{in}}{\nu}  & \text{on }\Gamma^\pm_j, \; j =1,2,\ldots,N.
\end{cases}
\end{equation}

\begin{prop}
Any solution $u$ to the system \eqref{D2N} can be written as
\[
u(x) = a_i j_0(k_\rmr r) + b_i h_0^{(1)}(k_\rmr r),\quad x\in D_i,\;i = 1,2,\ldots,N,
\]
where the constants satisfy
\begin{equation}\label{dtnls}
\mathbb{A}(\omega,\delta)[\bm{a}] = \delta(- \mathcal{T}^k[u^{in}]_1^+ +  \ddp{u^{in}}{\nu}|_{\Gamma_1^+},- \mathcal{T}^k[u^{in}]_1^- +  \ddp{u^{in}}{\nu}|_{\Gamma_1^-},\ldots,- \mathcal{T}^k[u^{in}]_N^- +  \ddp{u^{in}}{\nu}|_{\Gamma_N^-})^T.
\end{equation}
Here $\bm{a}: = (a_1,b_1,a_2,b_2,\ldots,a_N,b_N)^T$ and the $2N\times 2N$ matrix $\mathbb{A}(\omega,\delta)$ is given by
\begin{equation}\label{DtNmap}
\begin{aligned}
\mathbb{A}(\omega,\delta): = -k_\rmr\diag (A'_1,A'_2,\ldots,A'_N)-\delta \mathcal{T}^k\diag (A_1,A_2,\ldots,A_N)
\end{aligned}
\end{equation}
with
\[
A'_j = \begin{pmatrix}
\ds j_1(k_\rmr r_j^+)&  \ds h_1^{(1)}(k_\rmr r_j^+)\\
\nm
\ds j_1(k_\rmr r_j^-) &  \ds h_1^{(1)}(k_\rmr r_j^-)
\end{pmatrix}
\;\text{ and }\;
A_j = \begin{pmatrix}
\ds j_0(k_\rmr r_j^+)&  \ds h_0^{(1)}(k_\rmr r_j^+)\\
\nm
\ds j_0(k_\rmr r_j^-) &  \ds h_0^{(1)}(k_\rmr r_j^-)
\end{pmatrix}
\;\text{ for } j=1,2,\ldots,N.
\]
\end{prop}
\begin{proof}[\bf Proof]
It follows from the boundary condition of \eqref{D2N} that for $i = 1,2,\ldots,N$,
\[
k_\rmr\(a_i j_0'(k_\rmr r_i^\pm) + b_i h_0^{(1)'}(k_\rmr r_i^\pm)\)-\delta \mathcal{T}^k[u]_i^\pm = -\delta \mathcal{T}^k[u^{in}]_i^\pm + \delta \ddp{u^{in}}{\nu}|_{\Gamma_i^\pm}
\]
This, together with $f'_0(t) = -f_1(t)$ for both $f_n = j_n$ and $f_n = h_n^{(1)}$ when $n=1,2$, implies that  \eqref{dtnls} holds.
\end{proof}

\begin{prop}\label{SCR}
	subwavelength resonant frequencies $\omega=\omega(\delta)\in\mathbb{C}$ are determined by $\det(\mathbb{A}(\omega(\delta),\delta))=0$.
\end{prop}

So far, we reduce the $4N\times 4N$ problem \eqref{4Nby4N} to a smaller $2N\times 2N$ problem \eqref{dtnls} by using  the DtN map, which gives a second characterization of the resonances. However, it also seem rather tedious to compute asymptotic expansions of the determinant $\det(\mathbb{A}(\omega(\delta),\delta))$. From Lemma \ref{symmetric_res} and Theorem \ref{thmres}, we know that a structure of $N$-layer nested resonators  processes $2N$ eigenfrequencies that  are symmetric about the imaginary axis. Inspired by this fact, in the next subsection, we will continue the dimensionality reduction process using the DtN approach \cite{FA_JMPA2024}, which is crucial in approximating the subwavelength eigenfrequencies and eigenmodes of the MLHC structure with the eigenvalues of an $N\times N$ tridiagonal matrix, which we refer to as the capacitance matrix.

\subsection{A third characterization of resonances based on the DtN approach}

In this subsection, we use the DtN approach \cite{FA_JMPA2024}, which provides a systematic and insightful characterization of subwavelength resonances. This approach not only simplifies the analysis but also enhances the clarity and efficiency of the underlying physical models, making it an invaluable tool in the study of complicated structure.

By using variational principle, \eqref{D2N} can be formulated as the following weak form:
\begin{equation}\label{eq:varpro}
a(u,v) = \langle f,v\rangle_{H^{-1}(D),H^1(D)},  \text{ for }\forall v\in H^1(D),
\end{equation}
where the bilinear form $a(u,v)$ for $u,v\in H^{1}(D) $ is defined by
\[
a(u,v):=
\sum_{i=1}^N \int_{D_i}\left(\nabla u\cdot\nabla \overline{v} - \frac{\omega^2}{v_\rmr^2} u \overline{v}\right) ~\d x  - \delta\sum_{i=1}^N  \left(\int_{\Gamma_i^+}
\overline{v}\, \mathcal{T}^{k}[u]_i^+~\d \sigma  -\int_{\Gamma_i^-}
\overline{v}\, \mathcal{T}^{k}[u]_i^- ~\d \sigma \right),
\]
and
\begin{equation}\label{rightlinear}
\langle f,v\rangle_{H^{-1}(D),H^1(D)} := \delta \sum_{i=1}^N \left(\int_{\Gamma_i^+} \overline{v}\, \(-\mathcal{T}^{k}[u^{in}]_i^+ + \ddp{u^{in}}{\nu}\)~\d \sigma - \int_{\Gamma_i^-} \overline{v}\, \(-\mathcal{T}^{k}[u^{in}]_i^- + \ddp{u^{in}}{\nu}\)~\d \sigma\right).
\end{equation}
Next, we introduce a new bilinear form
\begin{equation}\label{NBF}
a_{\omega,\delta}(u,v) := a(u,v)+\sum_{i=1}^N \(\int_{D_i}u~\d x \int_{D_i}\overline{v}~\d x\).
\end{equation}
For sufficiently small $\delta$ and $\omega$, $a_{\omega, \delta}$ remains coercive as a analytic perturbation of the continuous coercive bilinear form  \cite{FA_JMPA2024}
\[
a_{0,0}(u,v) = \sum_{i=1}^N \int_{D_i}\nabla u\cdot\nabla \overline{v}  ~\d x  +\sum_{i=1}^N \(\int_{D_i}u~\d x \int_{D_i}\overline{v}~\d x\).
\]
Therefore, for any $f\in H^{-1}(D)$, there exists a
unique Lax-Milgram solution $u_f(\omega ,\delta )$ to the problem
\begin{equation}\label{Lax_Milgram_solu}
a_{\omega,\delta}(u_f(\omega ,\delta ),v) = \langle f,v\rangle_{H^{-1}(D),H^1(D)},  \text{ for }\forall v\in H^1(D),
\end{equation}
which is analytic with respect to $\omega$ and $\delta$.
In order to characterize the subwavelength resonances, we introduce $u_j(\omega,\delta)$ satisfying the following variational problems
\begin{equation}\label{eqn:9d5p8}
a_{\omega,\delta}(u_j(\omega,\delta),v)=\int_{D_j}\bar v\d x,\;   \text{ for }\forall v\in H^1(D),\; j=1,2,\ldots,N.
\end{equation}

In the following lemma we show that the functions $u_j(\omega,\delta)$ facilitate a reduction of the $2N\times 2N$ system \eqref{dtnls} to a smaller $N\times N$ linear system.  This reduction significantly simplifies the analysis, and from a cardinality point of view, it is optimal in the sense that it captures the necessary degrees of freedom while reducing the computational complexity, thereby allowing for a more efficient study for the resonant behavior of $N$-layer nested resonators.

\begin{lem}\label{lem32}
	Let $\omega \in \mathbb{C}$ and $\delta \in \mathbb{R}$ belong to a neighborhood of zero.  For any  $f\in H^{-1}(D)$, the variational problem \eqref{eq:varpro} has a unique solution $u:= u(\omega ,\delta )$ if and only if the following $N\times N$ linear system
	\begin{equation}\label{eq654}
	 (\bm{I} - \bm{U}(\omega,\delta))\bm{x}=\bm{F}
	\end{equation}
	has a unique solution $\bm{x}:=\(x_i(\omega,\delta)\)_{1\leq i \leq N}$, where $\bm{U}(\omega,\delta)$ and $\bm{F}$ are, respectively, given by
	\begin{equation}\label{eqn:hnkcq}
	\bm{U}(\omega,\delta):= (\bm{U}_{ij}(\omega,\delta))_{1\leq i,j\leq N}:=
	\left(\int_{D_i}u_j~\d x\right)_{1\leq i,j\leq N},   \end{equation}
	and
	\begin{equation}\label{right_linearsystem}
	\bm{F}:= ({F}_{i})_{1\leq i\leq N}: =
	\left(\int_{D_i}u_f~\d x\right)_{1\leq i\leq N}.
	\end{equation}
Moreover, the solution to \eqref{eq:varpro} can be given by
	\begin{equation}\label{eq656}
	u(\omega ,\delta )=u_f(\omega ,\delta ) + \sum_{j=1}^N x_j(\omega ,\delta )u_j(\omega ,\delta ).
	\end{equation}
\end{lem}

\begin{proof}[\bf Proof]
	It follows from \eqref{NBF} and \eqref{Lax_Milgram_solu} that
	the variational problem \eqref{eq:varpro}  is equivalent to
	\[
	a_{\omega,\delta}(u,v)-\sum_{i=1}^N \(a_{\omega,\delta}(u_i,v)\int_{D_i}u~\d x \)=a_{\omega,\delta}(u_f(\omega ,\delta ),v),
	\]
	which implies that
	\begin{equation}\label{eqn:q48jq}
		u-\sum_{j=1}^N \(u_j\int_{D_j}u~\d x \)=u_f(\omega ,\delta ).
	\end{equation}
	By integrating both sides of \eqref{eqn:q48jq} on $D_i$, we see that the vector $\bm{x}:=\(\int_{D_i}u~\d x\)_{1\leq i \leq N}$ solves the linear system
	\[
	\int_{D_i}u ~\d x -\sum_{j=1}^N \(\int_{D_i} u_j ~\d x \int_{D_j}u~\d x  \)= \int_{D_i} u_f(\omega ,\delta )~\d x,\; i=1,2,\ldots,N,
	\]
	which is exactly \eqref{eq654}. Conversely, if \eqref{eq654} has a solution, then \eqref{eqn:q48jq} implies that the solution to \eqref{eq:varpro} is given by \eqref{eq656}.
\end{proof}

Therefore, we reduce the $2N\times 2N$ system \eqref{dtnls} to the smaller $N\times N$ linear system \eqref{eq654}, and the third characterization of resonances is given in the following.

\begin{prop}\label{TCR}
	Subwavelength resonant frequencies $\omega=\omega(\delta)\in\mathbb{C}$ are determined by $\det(\bm{I} - \bm{U}(\omega,\delta))=0$.
\end{prop}

\subsection{Asymptotic expansions of the subwavelength resonances}

In this subsection, we give the exact formulas of $2N$ eigenfrequencies of any finite $N$-layer nested resonators.  Specifically, we compute their leading order and higher order asymptotic expansions. The solutions $u_j(\omega ,\delta )$, $j=1,2,\ldots,N$,  to variational problem  \eqref{eqn:9d5p8} will be determined for sufficiently small $\delta$ and $\omega$. Furthermore, based on Proposition \ref{TCR}, the eigenfrequencies will be derived. To this end, we consider the strong form of the variational problem \eqref{eqn:9d5p8}
\begin{equation}\label{SFVP}
\begin{cases}
\ds -\Delta u_j - \frac{\omega^2}{v_\rmr^2} u_j + \sum_{i=1}^{N}
\left(\int_{D_i} u_j ~\d x\right) \chi_{D_i} =
\chi_{D_j} &  \text{ in } D,\\
\nm
\ds \ddp{u_j}{\nu}|_\pm = \delta \mathcal{T}^k[u_j]_i^\mp  & \text{on }\Gamma^\mp_i,\; i =1,2,\ldots,N.
\end{cases}
\end{equation}

\begin{prop}\label{prop33}
	The solutions $u_j(\omega, \delta)$, $j=1,2,\ldots,N$,  to the variational problem \eqref{eqn:9d5p8} have the following asymptotic behaviour as $\omega,\delta \to 0:$
	\begin{equation}\label{asy_ujwd}
	\begin{aligned}
	u_j(\omega,\delta) &= \(\frac{1}{ |D_j|}+\frac{\omega^2}{v_\rmr^{2} |D_j|^{2}}\)\chi_{D_j}\\
	&\quad +\delta\left[\frac{4\pi r_{j-1}^{-} r_{j}^+}{|D_{j-1}|^2|D_j| (r_{j-1}^--r_{j}^+)}\chi_{D_{j-1}}-\(\frac{4\pi r_j^+ r_{j-1}^-}{|D_j|^3 (r_{j-1}^--r_{j}^+)}  +\frac{4\pi r_j^- r_{j+1}^+}{|D_j|^3 (r_j^--r_{j+1}^+)}\)\chi_{D_{j}}\right.\\
	&\quad \left.\quad\quad+\frac{4\pi r_{j+1}^+ r_{j}^-}{|D_{j+1}|^2 |D_j| (r_{j}^--r_{j+1}^+)} \chi_{D_{j+1}} + \widetilde{u}_{j,0,1}\right]\\
	&\quad +\omega\delta\(\frac{4\pi (r_j^+)^2\i}{|D_j|^{3}v
	}\delta_{1,j}\chi_{D_j}+\widetilde{u}_{j,1,1}\) +O((\omega^2+\delta)^2),
	\end{aligned}
	\end{equation}
	where $|D_j| = \frac{4\pi}{3}\((r_j^+)^3-(r_j^-)^3\)$, the functions $\widetilde{u}_{j,0,1}$ and $\widetilde{u}_{j,1,1}$  satisfy
	\[
	\int_{D_i}\widetilde{u}_{j,0,1}\d x=0,\quad  \int_{D_i}\widetilde{u}_{j,1,1}\d x=0,\;
	i = 1,2,\ldots, N.
	\]

\end{prop}
\begin{proof}[\bf Proof]
It follows from \eqref{coranalytic_DtN} that  there exist functions $(u_{j,m,n})_{m{\geq}0,	n{\geq}0}$ such that $u_j(\omega,\delta)$ has the following convergent series in $H^1(D)$:
\begin{equation}\label{ujcs}
u_j(\omega,\delta) = \sum_{m,n=0}^{+\infty} \omega^{m}\delta^n
u_{j,m,n}.
\end{equation}
Substituting \eqref{ujcs} into \eqref{SFVP} and identifying powers of $\omega$ and
$\delta$, we obtain the following system of equations:
\begin{equation}\label{ujmn}
\begin{cases}
\ds -\Delta u_{j,m,n}  + \sum_{i=1}^{N}
\left(\int_{D_i} u_{j,m,n} ~\d x\right) \chi_{D_i} =  \frac{1}{v_\rmr^2} u_{j,m-2,n}+
\chi_{D_j} \delta_{m,0}\delta_{n,0}&  \text{ in } D,\\
\nm
\ds \ddp{u_{j,m,n}}{\nu}|_\pm =  \sum_{p=0}^{m}\frac{1}{v^p}\mathcal{T}_p[u_{j,m-p,n-1}]_i^\mp  & \text{on }\Gamma^\mp_i,\; i =1,2,\ldots,N,
\end{cases}
\end{equation}
where $u_{j,m,n}=0$ for negative indices $m$ and $n$.
It can then be proved inductively that
\begin{equation}\label{ujevenodd0}
u_{j,2m,0} = \frac{\chi_{D_j}}{v_\rmr^{2m} |D_j|^{m+1}}\text{ and }u_{j,2m+1,0} = 0 \text{ for	any } m \geq 0.
\end{equation}

For $m=0$, $n=1$, we obtain the following system of equations
\begin{equation}\label{uj01}
\begin{cases}
\ds -\Delta u_{j,0,1}  + \sum_{i=1}^{N}
\left(\int_{D_i} u_{j,0,1} ~\d x\right) \chi_{D_i} =  0&  \text{ in } D,\\
\nm
\ds \ddp{u_{j,0,1}}{\nu}|_\pm =  \mathcal{T}_0[u_{j,0,0}]_i^\mp & \text{on }\Gamma^\mp_i,\; i =1,2,\ldots,N.
\end{cases}
\end{equation}
From \eqref{T_0expre} with $f_i^\pm = u_{j,0,0}|_{\Gamma_i^\pm}=\frac{\delta_{i,j}}{|D_j|}$, we obtain
\begin{equation}
\begin{cases}
\ds\mathcal{T}_0[u_{j,0,0}]_1^{+} = -\frac{1}{r_1^+}\frac{\delta_{1,j}}{|D_j|}, &  \\
\nm
\ds   \mathcal{T}_0[u_{j,0,0}]_i^{+} = \frac{r_{i-1}^-}{|D_j|r_{i}^+ (r_{i-1}^--r_{i}^+)}  \({\delta_{i-1,j}}-{\delta_{i,j}}\), &i =2,3,\ldots,N, \\
\nm
\ds   \mathcal{T}_0[u_{j,0,0}]_i^{-} = \frac{r_{i+1}^+}{|D_j|r_i^- (r_i^--r_{i+1}^+)} \({\delta_{i,j}}-{\delta_{i+1,j}}\), & i =1,2,\ldots,N-1,  \\
\nm
\ds  \mathcal{T}_0[u_{j,0,0}]_N^{-} = 0.&
\end{cases}
\end{equation}
Multiplying \eqref{uj01} by $\chi_{D_i}$ and integrating by parts, we have that
\begin{equation}
\begin{aligned}
\int_{D_i} u_{j,0,1}~\d x &= \frac{1}{|D_i|}\(\int_{\Gamma_i^+} \mathcal{T}_0[u_{j,0,0}]_i^{+}~\d \sigma - \int_{\Gamma_i^-} \mathcal{T}_0[u_{j,0,0}]_i^{-}~\d \sigma\)\\
& = \frac{1}{|D_i|}\(-\frac{4\pi r_1^+\delta_{1,j}}{|D_j|} \chi_{\{ i=1\}} + \frac{4\pi r_i^{+}r_{i-1}^-}{|D_j| (r_{i-1}^--r_{i}^+)}  \({\delta_{i-1,j}}-{\delta_{i,j}}\)\chi_{\{2\leq i\leq N\}}\right.\\
&\left.\quad \quad\quad\;
-\frac{4\pi r_i^{-} r_{i+1}^+}{|D_j| (r_i^--r_{i+1}^+)} \({\delta_{i,j}}-{\delta_{i+1,j}}\)\chi_{\{1\leq i\leq N-1\}}\)\\
& =
\begin{cases}
\ds \frac{4\pi r_{j-1}^{-}r_{j}^+}{|D_{j-1}||D_j| (r_{j-1}^--r_{j}^+)} &\text{ if } i=j-1,\\
\nm
\ds -\frac{4\pi r_j^{+} r_{j-1}^-}{|D_j|^2 (r_{j-1}^--r_{j}^+)}  -\frac{4\pi r_j^{-}r_{j+1}^+}{|D_j|^2 (r_j^--r_{j+1}^+)}  &\text{ if }  i=j,  \\
\nm
\ds \frac{4\pi r_{j+1}^{+}r_{j}^-}{|D_{j+1}||D_j| (r_{j}^--r_{j+1}^+)} &\text{ if } i=j+1,
\end{cases}
\end{aligned}
\end{equation}
where $r_{N+1}^+ = 0$ and $r_0^- \to  +\infty$. It means that
\[
-\frac{4\pi r_j^+ r_{j-1}^-}{|D_j|^2 (r_{j-1}^--r_{j}^+)}  -\frac{4\pi r_j^- r_{j+1}^+}{|D_j|^2 (r_j^--r_{j+1}^+)} =
\begin{cases}
\ds -\frac{4\pi r_1^+}{|D_1|^2 } -\frac{4\pi r_1^- r_{2}^+}{|D_1|^2 (r_1^--r_{2}^+)}&\text{ if } i=j = 1,\\
\nm
\ds -\frac{4\pi r_N^+ r_{N-1}^-}{|D_N|^2  (r_{N-1}^--r_{N}^+)}   &\text{ if } i=j = N.\\
\end{cases}
\]
It follows that $u_{j,0,1}$ is given by
\[
\begin{aligned}
u_{j,0,1} &=  \frac{4\pi r_{j-1}^- r_{j}^+}{|D_{j-1}|^2|D_j| (r_{j-1}^--r_{j}^+)}\chi_{D_{j-1}}-\(\frac{4\pi r_j^+ r_{j-1}^-}{|D_j|^3 (r_{j-1}^--r_{j}^+)}  +\frac{4\pi r_j^- r_{j+1}^+}{|D_j|^3  (r_j^--r_{j+1}^+)}\)\chi_{D_{j}}\\
&\quad +\frac{4\pi r_{j+1}^+ r_{j}^-}{|D_{j+1}|^2 |D_j| (r_{j}^--r_{j+1}^+)} \chi_{D_{j+1}} + \widetilde{u}_{j,0,1},
\end{aligned}
\]
where $\widetilde{u}_{j,0,1}$  satisfies $\int_{D_i}\widetilde{u}_{j,0,1}\d x=0$ for $i = 1,2,\ldots, N$.

For $m=n=1$, it folllows from \eqref{ujmn} and \eqref{ujevenodd0} that $u_{j,1,1}$ satisfies
\begin{equation}\label{uj11}
\begin{cases}
\ds -\Delta u_{j,1,1}  + \sum_{i=1}^{N}
\left(\int_{D_i} u_{j,1,1} ~\d x\right) \chi_{D_i} =  0&  \text{ in } D,\\
\nm
\ds \ddp{u_{j,1,1}}{\nu}|_\pm =  \frac{1}{v}\mathcal{T}_1[u_{j,0,0}]_i^\mp  & \text{on }\Gamma^\mp_i,\; i =1,2,\ldots,N.
\end{cases}
\end{equation}
By \eqref{T_0}, we have that
\[
\mathcal{T}_1[u_{j,0,0}]_i^{+} = \begin{cases}
\i \frac{\delta_{1,j}}{|D_j|}&\text{ if }i=1,\\
0&\text{ if }i{\geq}2,
\end{cases}
\text{ and } \mathcal{T}_1[u_{j,0,0}]_i^{-} = 0 \text{  for } i=1,2,\ldots,N.
\]
Therefore, multiplying \eqref{uj11} by $\chi_{D_i}$ and integrating by parts, we have that
\[
    \int_{D_i} u_{j,1,1}\d x=\i\frac{4\pi (r_i^+)^2\delta_{1,j}\delta_{1,i}}{|D_i| |D_j|  v},
\]
which implies  that
$u_{j,1,1}=\frac{ 4\pi (r_j^+)^2\i}{|D_j|^{3}v
}\delta_{1,j}\chi_{D_j}+\widetilde{u}_{j,1,1}$,
where $\widetilde{u}_{j,1,1}$  satisfies $\int_{D_i}
\widetilde{u}_{j,1,1}\d x=0$ for any $1\leq i\leq N$. The proof is complete.
\end{proof}

Next, we define the  capacitance matrix similar to the three-dimensional separated structure case \cite{FA_SAM_2022,AD_SIIS2023}.
\begin{defn}\label{def: Capacitance matrix}
	Consider the solutions $V_j : \RR^3\to \RR$ of the problem
	\begin{equation}\label{V_i}
	\begin{cases}
	-\Delta V_j(x) =0, & x\in \RR^3 \setminus D, \\
	V_j(x)=\delta_{ij}, & x\in D_i ,\\
	V_j(x) = O(|x|^{-1}), & \text{ as }|x|\to\infty;
	\end{cases}
	\end{equation}
	{for $1\leq i,j\leq N$}. Then the capacitance matrix is defined coefficient-wise by
	\begin{equation}\label{cap_matrix}
	\mathcal{C}_{ij} = -\(\int_{\Gamma_i^+}\frac{\partial V_j}{\partial \nu}\d \sigma- \int_{\Gamma_i^-}\frac{\partial V_j}{\partial \nu}\d \sigma\).
	\end{equation}
\end{defn}

\begin{lem}
	The capacitance matrix for a structure of $N$-layer nested resonators in $\RR^3$ is given by
	\[
	\ds \mathcal{C}_{ij}: = -4\pi \frac{r_{j-1}^{-} r_{j}^{+}}{r_{j-1}^{-}-r_{j}^{+}}\delta_{i,j-1}
	+4\pi\(\frac{r_{j-1}^{-} r_{j}^{+}}{r_{j-1}^{-}- r_{j}^{+}} + \frac{r_{j}^{-} r_{j+1}^{+}}{r_{j}^{-}- r_{j+1}^{+}}\)\delta_{i,j}
	-4\pi \frac{r_{j}^{-} r_{j+1}^{+}}{r_{j}^{-}-r_{j+1}^{+}}\delta_{i,j+1},
	\]
	where $r_{N+1}^+ = 0$ and $r_0^- \to +\infty$. That is
	\begin{equation}\label{capmatrix}
	\small{
		\begin{split}
		&\mathcal{C}:=4\pi \times \\
		&\begin{pmatrix}
		\ds r_{1}^{+} + \frac{r_{1}^{-} r_{2}^{+}}{r_{1}^{-}- r_{2}^{+}} &  \ds - \frac{r_{1}^{-} r_{2}^{+}}{r_{1}^{-}-r_{2}^{+}} & &  & & \\
		\nm
		\ds -\frac{r_{1}^{-} r_{2}^{+}}{r_{1}^{-}-r_{2}^{+}} & \ds \frac{r_{1}^{-} r_{2}^{+}}{r_{1}^{-}- r_{2}^{+}} + \frac{r_{2}^{-} r_{3}^{+}}{r_{2}^{-}- r_{3}^{+}} & \ds - \frac{r_{2}^{-} r_{3}^{+}}{r_{2}^{-}-r_{3}^{+}}&  & &\\
		\nm
		& \ds-\frac{r_{2}^{-} r_{3}^{+}}{r_{2}^{-}- r_{3}^{+}} & \ds \frac{r_{2}^{-} r_{3}^{+}}{r_{2}^{-}- r_{3}^{+}} + \frac{r_{3}^{-} r_{4}^{+}}{r_{3}^{-}- r_{4}^{+}} & \ds-\frac{r_{3}^{-} r_{4}^{+}}{r_{3}^{-}- r_{4}^{+}} & & \\
		\nm
		& & \ddots &\ddots & \ddots& \\
		\nm
		&  & &\ds-\frac{r_{N-2}^{-} r_{N-1}^{+}}{r_{N-2}^{-}- r_{N-1}^{+}} &\ds \frac{r_{N-2}^{-} r_{N-1}^{+}}{r_{N-2}^{-}- r_{N-1}^{+}} + \frac{r_{N-1}^{-} r_{N}^{+}}{r_{N-1}^{-}- r_{N}^{+}} &\ds -\frac{r_{N-1}^{-} r_{N}^{+}}{r_{N-1}^{-}- r_{N}^{+}}\\
		\nm
		&  & &  &\ds-\frac{r_{N-1}^{-} r_{N}^{+}}{r_{N-1}^{-}- r_{N}^{+}} &\ds \frac{r_{N-1}^{-} r_{N}^{+}}{r_{N-1}^{-}- r_{N}^{+}}
		\end{pmatrix}.
		\end{split}}
	\end{equation}
\end{lem}
\begin{proof}[\bf Proof]
	The solutions $V_j$, $j=1,2,\ldots,N$, to \eqref{V_i}  are given by
\begin{equation}\label{V_i_solu}
V_j(x) =
\begin{cases}
\ds -\frac{r_j^{-} r_{j+1}^{+}}{r_j^{-}-r_{j+1}^{+}} \frac{1}{|x|} + \frac{r_j^{-} }{r_j^{-}-r_{j+1}^{+}}, &r_{j+1}^{+}\leq |x|\leq r_j^{-},\\
1, & r_{j}^{-}\leq |x|\leq r_j^{+},\\
\ds \frac{r_{j-1}^{-} r_{j}^{+}}{r_{j-1}^{-}-r_{j}^{+}} \frac{1}{|x|} - \frac{r_j^{+} }{r_{j-1}^{-}-r_{j}^{+}}, & r_{j}^{+}\leq |x|\leq r_{j-1}^{-},\\
0, & \text{else},
\end{cases}
\end{equation}
where $r_{N+1}^+ = 0$ and $r_0^- \to  +\infty$.  From definition \eqref{cap_matrix}, we have
\[
\begin{aligned}
\mathcal{C}_{ij} &= \delta_{i,j} \frac{r_j^{-} r_{j+1}^{+}}{r_j^{-}-r_{j+1}^{+}}\int_{\Gamma_i^-}  \frac{1}{|x|^2}~\d \sigma(x) - \delta_{i,j-1}\frac{r_{j-1}^{-} r_{j}^{+}}{r_{j-1}^{-}-r_{j}^{+}} \int_{\Gamma_i^-}  \frac{1}{|x|^2}~\d \sigma(x)\\
&\quad -\(-\delta_{i,j} \frac{r_{j-1}^{-} r_{j}^{+}}{r_{j-1}^{-}-r_{j}^{+}}\int_{\Gamma_i^+}  \frac{1}{|x|^2}~\d \sigma(x) + \delta_{i,j+1}\frac{r_j^{-} r_{j+1}^{+}}{r_j^{-}-r_{j+1}^{+}} \int_{\Gamma_i^+}  \frac{1}{|x|^2}~\d \sigma(x)\)\\
& = -4\pi \frac{r_{j-1}^{-} r_{j}^{+}}{r_{j-1}^{-}-r_{j}^{+}}\delta_{i,j-1}
+4\pi\(\frac{r_{j-1}^{-} r_{j}^{+}}{r_{j-1}^{-}- r_{j}^{+}} + \frac{r_{j}^{-} r_{j+1}^{+}}{r_{j}^{-}- r_{j+1}^{+}}\)\delta_{i,j}
-4\pi \frac{r_{j}^{-} r_{j+1}^{+}}{r_{j}^{-}-r_{j+1}^{+}}\delta_{i,j+1}.
\end{aligned}
\]
The proof is complete.
\end{proof}

In what follows, we introduce the $N \times N$ matrices
\[
\bm{V} = \diag(|D_1|,|D_2|,\ldots,|D_N|) \text{ and } \bm{E}_1 = \diag(1,0,\ldots,0).
\]

\begin{cor}
	The following asymptotic expansion for the matrix $\bm{U}(\omega,\delta)$ defined in \eqref{eqn:hnkcq} holds:
	\begin{equation}\label{bmU}
	\bm{U}(\omega,\delta) = \bm{I}+\frac{\omega^2}{v_\rmr^{2}}\bm{V}^{-1}-\delta \bm{V}^{-1}\mathcal{C} \bm{V}^{-1} +4\pi (r_1^+)^2\frac{\i \omega\delta}{v} \bm{V}^{-1}\bm{E}_1 \bm{V}^{-1}+O((\omega^2+\delta)^2).
	\end{equation}
\end{cor}

\begin{proof}[\bf Proof]
It follows from \eqref{asy_ujwd} and \eqref{eqn:hnkcq} that
	\[
	\begin{aligned}
	\bm{U}_{ij}(\omega,\delta) &= \(1+\frac{\omega^2}{v_\rmr^{2} |D_i|}\)\delta_{i,j}\\
	&\quad +\delta\left[\frac{4\pi r_{j-1}^- r_{j}^+}{|D_{i}||D_j| (r_{j-1}^--r_{j}^+)}\delta_{i,j-1}-\(\frac{4\pi r_j^+ r_{j-1}^-}{|D_i| |D_j|  (r_{j-1}^--r_{j}^+)}  +\frac{4\pi r_j^- r_{j+1}^+}{|D_i||D_j| (r_j^--r_{j+1}^+)}\)\delta_{i,j}\right.\\
	&\quad \left.\quad\quad+\frac{4\pi r_{j+1}^+|r_{j}^-}{|D_{i}| |D_j| (r_{j}^--r_{j+1}^+)} \delta_{i,j+1} \right]\\
	&\quad +\frac{4\pi (r_j^+)^2\i \omega\delta}{|D_j|^{2}v
	}\delta_{1,j}\delta_{i,j} +O((\omega^2+\delta)^2).
	\end{aligned}
	\]
	The proof is complete.
\end{proof}
The next lemma shows that the capacitance matrix $\mathcal{C}$ is  positive-definite.

\begin{lem}\label{CMSP}
The capacitance matrix $\mathcal{C}$ defined in \eqref{capmatrix} is positive-definite.
\end{lem}
\begin{proof}[\bf Proof]
For any vector $\bm{d} = (d_i)_{1\le i\le N}\in \RR ^N$,
\[
\begin{aligned}
\bm{d}^T \mathcal{C}  \bm{d}& = 4\pi\sum_{j=1}^N \left[- \frac{r_{j-1}^{-} r_{j}^{+}}{r_{j-1}^{-}-r_{j}^{+}}d_{j-1}d_j
+\(\frac{r_{j-1}^{-} r_{j}^{+}}{r_{j-1}^{-}- r_{j}^{+}} + \frac{r_{j}^{-} r_{j+1}^{+}}{r_{j}^{-}- r_{j+1}^{+}}\)d_j^2
- \frac{r_{j}^{-} r_{j+1}^{+}}{r_{j}^{-}-r_{j+1}^{+}}d_j d_{j+1}\right]\\
& = 4\pi\sum_{j=1}^{N-1} \left[- \frac{r_{j}^{-} r_{j+1}^{+}}{r_{j}^{-}-r_{j+1}^{+}}d_{j}d_{j+1}
+\frac{r_{j}^{-} r_{j+1}^{+}}{r_{j}^{-}- r_{j+1}^{+}}d_{j+1}^2 + \frac{r_{j}^{-} r_{j+1}^{+}}{r_{j}^{-}- r_{j+1}^{+}}d_j^2
- \frac{r_{j}^{-} r_{j+1}^{+}}{r_{j}^{-}-r_{j+1}^{+}}d_j d_{j+1}\right] +4\pi r_1^{+} d_1^2\\
 & =  4\pi\sum_{j=1}^{N-1} \frac{r_{j}^{-} r_{j+1}^{+}}{r_{j}^{-}-r_{j+1}^{+}}\(d_{j+1}-d_j\)^2+4\pi r_1^{+} d_1^2,
\end{aligned}
\]
where we have used $d_0 = d_{N+1} = 0$, $r_{N+1}^+ = 0$ and $r_0^- \to  +\infty$.
Noting that $\frac{r_{j}^{-} r_{j+1}^{+}}{r_{j}^{-}-r_{j+1}^{+}}>0$ for $j=1,2,\ldots,N-1$, we have that $\bm{d}^T \mathcal{C}  \bm{d}\geq 0 $ for any $\bm{d} \in \RR ^N$ with equality if and only if $\bm{d} = \bm{0}$. The proof is complete.
\end{proof}

Next, we consider the $N$ eigenvalues $(\lambda_i)_{1\le i\le N}$ and eigenvectors $(\bm{a}_i)_{1\le i\le N}$ of the generalized capacitance eigenvalue problem:
\begin{equation}\label{GEP}
\mathcal{C}\bm{a}_i = \lambda_i \bm{V}\bm{a}_i, \; i = 1,2,\ldots,N,
\end{equation}
where the eigenvectors form an orthonormal basis with respect to the following inner product
\begin{equation}\label{orthonormalV}
\bm{a}_i^T\bm{V}\bm{a}_j = \delta_{i,j}.
\end{equation}

 Combining Lemma \ref{CMSP} with \cite[Lemma 7.7.1]{SEP1998}, we have the following lemma.

\begin{lem}
	The $N$ eigenvalues of the capacitance matrix $\mathcal{C}$ in $\RR^3$ defined by \eqref{capmatrix} are distinct:
	\begin{equation}\label{eigenpositive}
	0<\lambda_1<\lambda_2<\cdots<\lambda_N.
	\end{equation}
\end{lem}

Our main result in this subsection is the following:

\begin{thm}\label{thmeigenfreqc}
The acoustic scattering problem \eqref{main_equation} in the MLHC structure $D = \cup_{j=1}^{N}D_{j}$  admits exactly $2N$ eigenfrequencies:
\begin{equation}\label{equeigenfreqc}
\omega_i^\pm(\delta) = \pm \delta^{\frac{1}{2}}\lambda_i^{\frac{1}{2}}v_\rmr -2\pi (r_1^+)^2 \frac{\i \delta  v_\rmr^2 }{ v} \bm{a}_i^T\bm{E}_1\bm{a}_i + O(\delta^{\frac{3}{2}})\quad \text{ for } i = 1,2,\ldots,N.
\end{equation}
\end{thm}
\begin{proof}[\bf Proof]
	We first let $\beta:= \frac{\omega^2}{\delta}$ and introduce the function
	\[
	F((\beta,\bm{x}),\omega):=\(\frac{\beta}{\omega^2}\(\bm{I}-\bm{U}(\omega,\delta)\)\bm{V}\bm{x},\bm{x}^T\bm{V}\bm{x}-1\).
	\]
	From \eqref{bmU}, we have
\begin{equation}\label{Veigenvector}
\frac{\beta}{\omega^2}\(\bm{I}-\bm{U}(\omega,\delta)\)\bm{V}\bm{x} = \(-\frac{\beta}{v_\rmr^{2}} \bm{I} + \bm{V}^{-1}\mathcal{C}  -4\pi (r_1^+)^2\frac{\i \omega}{v} \bm{V}^{-1}\bm{E}_1 +O(\omega^2)\)\bm{x}.
\end{equation}
Then it is easy to see that $F$ is a smooth function of $\omega,\beta\in \mathbb{C}$.
	By \eqref{GEP}, for $\omega=0$, it holds $F((\lambda_i v_\rmr^2,\bm{a}_i),0)=0$ for $i=1,2,\ldots,N$. In order to use the implicit function theorem, we next show that the differential of $(\beta,\bm{x})\mapsto F((\beta,\bm{x}),0)$ is invertible at $(\beta,\bm{x}) = (\lambda_i v_\rmr^2,\bm{a}_i)$. Through straightforward calculations, we have that for $(\tilde{\beta},\tilde{\bm{x}})\in \mathbb{C}\times\mathbb{C}^N$,
	\[
	\mathrm{D}F((\lambda_i v_\rmr^2,\bm{a}_i),0) (\tilde{\beta},\tilde{\bm{x}})^T =
	\begin{pmatrix}
	-\frac{\bm{a}_i}{v_\rmr^2}&-\lambda_i \bm{I}+\bm{V}^{-1}\mathcal{C}\\
	0&2\bm{a}_i^T\bm{V}
	\end{pmatrix}
	\begin{pmatrix}
	\tilde{\beta}\\
	\tilde{\bm{x}}
	\end{pmatrix} =
\begin{pmatrix}-\tilde{\beta}\frac{\bm{a}_i}{v_\rmr^2}+(-\lambda_i \bm{I}+\bm{V}^{-1}\mathcal{C})\tilde{\bm{x}}\\2\bm{a}_i^T\bm{V}\tilde{\bm{x}}\end{pmatrix}.
	\]
From \eqref{GEP}--\eqref{eigenpositive}, we know that $\mathrm{D}F((\lambda_i v_\rmr^2,\bm{a}_i),0)$ is invertible.
Thus, we get the existence of analytic functions $\beta_i(\omega)$ and $\bm{a}_i(\omega)$ satisfying
	\begin{equation}\label{analyba}
	F((\beta_i(\omega),\bm{a}_i(\omega)),\omega) = 0,
	\end{equation}
	 for $\omega\in \mathbb{C}$ belonging to a neighborhood of zero with $\beta_i(0) = \lambda_i v_\rmr^2$ and $\bm{a}_i(0) =  \bm{a}_i$. Furthermore, differentiating \eqref{analyba} with respect to $\omega$ at $\omega=0$, we have that
	\[
	\begin{cases}
	-\frac{\beta_i'(0)}{v_\rmr^{2}} \bm{a}_i   -4\pi (r_1^+)^2\frac{\i }{v} \bm{V}^{-1}\bm{E}_1\bm{a}_i +\(-\lambda_i \bm{I}+ \bm{V}^{-1}\mathcal{C}\)\bm{a}'_i(0) = 0,\\
	\bm{a}'_i(0)^T\bm{V}\bm{a}_i= 0.
	\end{cases}.
	\]
	Left multiplying by $\bm{a}_i^T\bm{V}$, we can obtain
	\[
	\beta_i'(0) = -4\pi (r_1^+)^2\frac{\i v_\rmr^2}{v} \bm{a}_i^T\bm{E}_1\bm{a}_i \quad \text{ and }\; \bm{a}'_i(0) =4\pi (r_1^+)^2\frac{\i}{v} \sum_{j\neq i}\frac{\bm{a}_j^T \bm{E}_1\bm{a}_i}{\lambda_j-\lambda_i}\bm{a}_j,
	\]
	which implies that
	\begin{equation}\label{beta}
	\beta_i(\omega) = \lambda_i v_\rmr^2-4\pi (r_1^+)^2\frac{\i \omega v_\rmr^2}{v} \bm{a}_i^T\bm{E}_1\bm{a}_i + O(\omega^2),
	\end{equation}
	and
	\begin{equation}\label{eigenvetor_omega}
	\bm{a}_i(\omega) = \bm{a}_i +  4\pi (r_1^+)^2\frac{\i\omega}{v} \sum_{j\neq i}\frac{\bm{a}_j^T \bm{E}_1\bm{a}_i}{\lambda_j-\lambda_i}\bm{a}_j +O(\omega^2).
	\end{equation}
	Eigenfrequencies  are the solutions to the equation $\omega^2=\delta \beta_i(\omega)$, that is,
	$
	\omega = \delta^{\frac{1}{2}}\sqrt{\beta_i(\omega)}$ or  $\omega = -\delta^{\frac{1}{2}}\sqrt{\beta_i(\omega)}.
	$
	 By using \eqref{beta}, we can obtain that
	\[
	\begin{aligned}
	\omega_i^\pm(\delta) &= \pm \delta^{\frac{1}{2}}\lambda_i^{\frac{1}{2}}v_\rmr \sqrt{1\mp 4\pi (r_1^+)^2\frac{\i \delta^{\frac{1}{2}} \lambda_i^{\frac{1}{2}} v_\rmr }{\lambda_i v} \bm{a}_i^T\bm{E}_1\bm{a}_i + O(\delta)}\\
	& = \pm \delta^{\frac{1}{2}}\lambda_i^{\frac{1}{2}}v_\rmr \({1\mp 2\pi (r_1^+)^2\frac{\i \delta^{\frac{1}{2}} \lambda_i^{-\frac{1}{2}} v_\rmr }{ v} \bm{a}_i^T\bm{E}_1\bm{a}_i + O(\delta)}\)\\
	& = \pm \delta^{\frac{1}{2}}\lambda_i^{\frac{1}{2}}v_\rmr -2\pi (r_1^+)^2 \frac{\i \delta  v_\rmr^2 }{ v} \bm{a}_i^T\bm{E}_1\bm{a}_i + O(\delta^{\frac{3}{2}}).
	\end{aligned}
	\]
	The proof is complete.
\end{proof}
%


\section{Modal decompositions and point scatterer approximations}\label{MSmonopolar}

In this section,  we shall provide the  modal decomposition and  point-scatterer approximation of the solution $u(\omega ,\delta )$ to the scattering problem \eqref{main_equation} in MLHC metamaterial $D = \cup_{j=1}^{N}D_{j}$ based on DtN  approach.

We first derive anasymptotic expansion of the solution $u_f(\omega,\delta)$ to the variational problem \eqref{Lax_Milgram_solu}.

\begin{prop}
Let $f\in H^{-1}(D)$ be given in \eqref{rightlinear}. The solution $u_f(\omega,\delta)$ to \eqref{Lax_Milgram_solu} has the following asymptotic expansion:
\begin{equation}\label{Lax_Milgram_solu_asym_expa}
u_f(\omega,\delta) = \delta \(\frac{4\pi r_1^+}{|D_1|^2}u^{in}(0)\chi_{D_1}+\widetilde{v}_{0,1}\)
+ O(\omega\delta),
\end{equation}
where 
$\int_{D_i}\widetilde{v}_{0,1}~\d x = 0$ for $i=1,2,\ldots,N$.
\end{prop}
\begin{proof}[\bf Proof]
We consider the following strong form of the variational problem \eqref{Lax_Milgram_solu}
\begin{equation} \label{Lax_Milgram_solu_strong}
\begin{cases}
\ds - \Delta u_f  - \frac{\omega^2}{v_\rmr^2} u_f + \sum_{i=1}^N \(\int_{D_i}u_f~\d x \)\chi_{D_i} = 0, & \text{in  } D, \\
\nm
\ds   \ddp{u_f}{\nu}|_\mp -\delta \mathcal{T}^k[u_f]_j^\pm =- \delta \mathcal{T}^k[u^{in}]_j^\pm + \delta\ddp{u^{in}}{\nu} , & \text{on }\Gamma^\pm_j, \; j =1,2,\ldots,N.
\end{cases}
\end{equation}
Using the Taylor expansion of plane wave $u^{in}$ at the origin, and the  fact that $\nabla u^{in} = O(\omega) $ and $\nabla^2 u^{in} = O(\omega^2)$, we have that on $\Gamma_j^\pm$
\[
\begin{aligned}
- \delta \mathcal{T}^k[u^{in}]_j^\pm + \delta\ddp{u^{in}}{\nu}& = \delta\left[-\(\mathcal{T}_0+\frac{\omega}{v}\mathcal{T}_1\)\left[u^{in}(0)+\nabla u^{in}(0)\cdot x\right]_j^\pm +O(\omega)\right]\\
& =
\begin{cases}
\ds \frac{u^{in}(0)}{r_1^+}\delta   + O(\omega\delta), &\text{ on } \Gamma_1^+,\\
\nm
\ds O(\omega\delta), &\text{ else }.
\end{cases}
\end{aligned}
\]
One can make the ansatz
\[
u_f(\omega,\delta) = \delta v_{0,1} + O(\omega\delta).
\]
 Substituting this expression into \eqref{Lax_Milgram_solu_strong} and identifying powers of $\omega$ and
 $\delta$, we obtain the following system for $v_{0,1}$:
 \[
 \begin{cases}
 \ds - \Delta v_{0,1}   + \sum_{i=1}^N \(\int_{D_i}v_{0,1}~\d x \)\chi_{D_i} = 0, & \text{in  } D \\
 \nm
 \ds   \ddp{v_{0,1}}{\nu}|_-  = \frac{u^{in}(0)}{r_1^+} , & \text{on }\Gamma^+_1, \\
 \nm
 \ds   \ddp{v_{0,1}}{\nu}|_- = 0, & \text{on }\Gamma^+_j,\; j =2,3,\ldots,N,  \\
 \nm
 \ds  \ddp{v_{0,1}}{\nu}|_+  =0, & \text{on }\Gamma^-_j,\; j =1,2,\ldots,N.
 \end{cases}
 \]
 Integrating by parts on $D_i$, we have
 \[
 \int_{D_i}v_{0,1}~\d x = \frac{4\pi r_i^+}{|D_i|}u^{in}(0)\delta_{1,i},
 \]
 which implies that
 \[
 v_{0,1} = \frac{4\pi r_1^+}{|D_1|^2}u^{in}(0)\chi_{D_1}+\widetilde{v}_{0,1},
 \]
 where  $\int_{D_i}\widetilde{v}_{0,1}~\d x = 0$ for $i=1,2,\ldots,N$.
\end{proof}

\begin{prop}
	For  $\omega\in \RR$ satisfying  $\omega = O(\delta^{\frac{1}{2}})$ and any  $\bm{F} \in \mathbb{C}^N$ in \eqref{eq654}, the solution $\bm{x}(\omega,\delta)$ to \eqref{eq654}  has the following asymptotic modal decomposition:
\begin{equation}\label{eq881}
\bm{x} (\omega,\delta) = -\sum_{i=1}^N \frac{v_\rmr^2}{ \omega^2-\delta\lambda_i v_\rmr^2+4\pi \i \omega\delta (r_1^+)^2\frac{ v_\rmr^2}{v} \bm{a}_i^T\bm{E}_1\bm{a}_i} (1+O(\delta^\frac{1}{2}))\bm{a}_i^T\bm{V}{\bm{F}} \bm{V}\bm{a}_i.
\end{equation}
\end{prop}
\begin{proof}[\bf Proof]
	Let $\beta:= \frac{\omega^2}{\delta}$, $\bm{y}: = \bm{V}^{-1}\bm{x}$, $\mathcal{G}(\omega,\beta):= \frac{\beta}{\omega^2}\(\bm{I}-\bm{U}(\omega,\frac{\omega^2}{\beta})\)\bm{V}$, and $\hat{\bm{F}}:=\frac{1}{\delta}\bm{F}$. Then  \eqref{eq654} can be rewritten by
	\begin{equation}\label{NNlinear}
	\mathcal{G}(\omega,\beta) \bm{y} = \hat{\bm{F}}.
	\end{equation}
	By continuity of the determinant, $(\bm{a}_i(\omega ))_{1\leq i\leq N}$ is a basis of $\mathbb{C}^N$ for $\omega$ sufficiently	small. This enables us to decompose $\bm{y} = \bm{y} (\omega,\delta)$ onto this basis with corresponding coefficients $(y_i(\omega,\delta))_{1\leq i\leq N}$:
	\begin{equation}\label{eq45}
	\bm{y} (\omega,\delta) = \sum_{i=1}^N y_i(\omega,\delta) \bm{a}_i(\omega ).
	\end{equation}
	In view of \eqref{analyba}, one has  $	\mathcal{G}(\omega,\beta_i(\omega)) \bm{a}_i(\omega ) = 0$ for $i=1,2,\ldots,N$. It follows that
	\begin{equation}\label{eq880}
	\mathcal{G}(\omega,\beta) \bm{y} = \sum_{i=1}^N y_i(\omega,\delta) (\mathcal{G}(\omega,\beta)- \mathcal{G}(\omega,\beta_i(\omega))) \bm{a}_i(\omega).
	\end{equation}
	By using \eqref{Veigenvector}, we have
	\begin{equation}\label{eqn47}
	(\mathcal{G}(\omega,\beta)- \mathcal{G}(\omega,\beta_i(\omega))) \bm{a}_i(\omega) = -\frac{\beta-\beta_i(\omega)}{v_\rmr^2}\bm{a}_i(\omega)(1+O(\omega^2)).
	\end{equation}
	Left multiplying \eqref{NNlinear} by $\bm{a}_j^T\bm{V}$ and using \eqref{eq880}--\eqref{eqn47}, we obtain
	\[
	-\sum_{i=1}^N  \bm{a}_j^T\bm{V}\bm{a}_i(\omega)
	 \frac{\beta-\beta_i(\omega)}{v_\rmr^2}(1+O(\omega^2))y_i(\omega,\delta) = \bm{a}_j^T\bm{V}\hat{\bm{F}}.
	\]
	 From \eqref{orthonormalV} and \eqref{eigenvetor_omega}, one has $\bm{a}_j^T\bm{V}\bm{a}_i(\omega) = \delta_{i,j}+O(\omega)$. It follows that
	 \[
	 (\beta-\beta_i(\omega))y_i(\omega,\delta) = -(1+O(\omega)){v_\rmr^2}\bm{a}_i^T\bm{V}\hat{\bm{F}}.
	 \]
	 This, together with \eqref{eigenvetor_omega} and \eqref{eq45}, implies that
	 \begin{equation}
	 \bm{y} (\omega,\delta) = -\sum_{i=1}^N \frac{1}{\beta-\beta_i(\omega)} (1+O(\omega)){v_\rmr^2}\bm{a}_i^T\bm{V}\hat{\bm{F}} \bm{a}_i.
	 \end{equation}
	 Therefore, we obtain that for any $\omega\in \mathbb{C}$,
	 \[
	 \bm{x} (\omega,\delta) = -\sum_{i=1}^N \frac{v_\rmr^2}{\omega^2-\delta\beta_i(\omega)} (1+O(\omega))\bm{a}_i^T\bm{V}{\bm{F}} \bm{V}\bm{a}_i.
	 \]
	 Furthermore, for $\omega\in \RR$ satisfying  $\omega = O(\delta^{\frac{1}{2}})$, it follows  from \eqref{beta} that
	 \[
	 \begin{aligned}
	 \omega^2-\delta\beta_i(\omega)& = \omega^2-\delta\lambda_i v_\rmr^2+4\pi \i \omega\delta (r_1^+)^2\frac{ v_\rmr^2}{v} \bm{a}_i^T\bm{E}_1\bm{a}_i + O(\omega^2\delta)\\
	 & = \(\omega^2-\delta\lambda_i v_\rmr^2+4\pi \i \omega\delta (r_1^+)^2\frac{ v_\rmr^2}{v} \bm{a}_i^T\bm{E}_1\bm{a}_i\)\(1 + O(\delta^{\frac{1}{2}})\),
	 \end{aligned}
	 \]
	 where the order $O(\delta^{\frac{1}{2}})$ can be derived by using a similar argument to the proofs of \cite[Proposition 4.2 and Lemma 4.3]{FA_SAM_2022}. The proof is complete.
\end{proof}

\begin{thm}\label{thm41}
For $\omega\in \RR$ satisfying  $\omega = O(\delta^{\frac{1}{2}})$ and a given plane wave $u^{in}$, the total field $u(\omega ,\delta )$ to the scattering problem \eqref{main_equation} has the following asymptotic modal decomposition in the resonator-nested $D = \cup_{j=1}^{N}D_{j}$ as $\delta \rightarrow 0$:
	\begin{equation}\label{modedecototalu}
	u(\omega,\delta ) = -4\pi r_1^+ u^{in}(0) \sum_{i=1}^N \frac{1}{\lambda_i}\frac{ { \bm{a}^{(1)}_{i}}}{\frac{\omega^2}{\omega_{M,i}^2}-1+\i \gamma_i  } (1+O(\delta^{\frac{1}{2}}))  \sum_{j=1}^N \bm{a}^{(j)}_{i} \chi_{D_j},
	\end{equation}
	where
	\begin{equation}\label{omegamm}
	\omega_{M,i}:=\delta^{\frac{1}{2}}\lambda_i^{\frac{1}{2}}v_\rmr,\; \gamma_i
	: = \frac{4\pi  (r_1^+)^2 \omega}{\lambda_iv} \(\bm{a}^{(1)}_{i}\)^2.
	\end{equation}
\end{thm}
\begin{proof}[\bf Proof]
	It follows from \eqref{right_linearsystem} and \eqref{Lax_Milgram_solu_asym_expa} that
	\[
	F_i = \delta \frac{4\pi r_1^+}{|D_1|}u^{in}(0)\delta_{i,1}
	+ O(\omega\delta),
	\]
	 which implies that
	 \[
	 \bm{a}_i^T\bm{V}\bm{F} =  {4\pi r_1^+}u^{in}(0)\bm{a}^{(1)}_{i}\delta
	 + O(\omega\delta).
	 \]
	 Substituting this expression into \eqref{eq881}, in  the subwavelength regime $\omega = \sqrt{\delta}$, we obtain the asymptotic expansions of the vector $\bm{x}:=\(x_j(\omega,\delta)\)_{1\leq j \leq N}$
	  \begin{equation}\label{bmx}
	  \begin{aligned}
	  x_j (\omega,\delta)& = -4\pi r_1^+ u^{in}(0)\sum_{i=1}^N \frac{\bm{a}^{(1)}_{i}\delta v_\rmr^2}{ \omega^2-\delta\lambda_i v_\rmr^2+4\pi \i \omega\delta (r_1^+)^2\frac{ v_\rmr^2}{v} \bm{a}_i^T\bm{E}_1\bm{a}_i} (1+O(\delta^\frac{1}{2}))|D_j|\bm{a}^{(j)}_i\\
	  & = -4\pi r_1^+ u^{in}(0)\sum_{i=1}^N \frac{1}{\lambda_i}\frac{\bm{a}^{(1)}_{i}}{ \frac{\omega^2}{\omega_{M,i}^2}-1+\i \gamma_i } (1+O(\delta^\frac{1}{2}))|D_j|\bm{a}^{(j)}_i.
	  \end{aligned}
	  \end{equation}
Substituting \eqref{asy_ujwd}, \eqref{Lax_Milgram_solu_asym_expa}, and \eqref{bmx} into \eqref{eq656} concludes the proof.
\end{proof}

In the following, we prove that the structure of $N$-layer nested resonators can be approximated as a point scatterer with monopole modes: $u^{s}(x)$ is approximately proportional to the fundamental solution
$G_k(x)$ given in \eqref{fundamentalk} as $|x| \to +\infty$.

\begin{thm} For $\omega\in \RR$ satisfying $\omega =O(\delta^{\frac{1}{2}})$ and a given plane wave $u^{in}$, the scattered field to \eqref{main_equation} can be approximated as  $\delta\to 0$ by
\[
u^{s}(x)  =  16\pi (r_1^+)^2 u^{in}(0)    \(\sum_{i=1}^N \frac{1}{\lambda_i}\frac{ \({ \bm{a}^{(1)}_{i}}\)^2}{\frac{\omega^2}{\omega_{M,i}^2}-1+\i \gamma_i }\) \(1+O(\delta^{\frac{1}{2}})+O(|x|^{-1})\) G_k(x),
\]	
when $|x|$ is sufficiently large, where  $\omega_{M,i}$ and $\gamma_i$ are difined in \eqref{omegamm}.
\end{thm}
\begin{proof}[\bf Proof]
	In view of \eqref{Helm_solution}, we have that in the regime $\omega\to 0 $,
	\begin{equation}\label{usc}
	u^{s}(x) = (u-u^{in})(x) =  \mathcal{S}_{\Gamma_1^+}^{k} [(\mathcal{S}_{\Gamma_1^+}^{k} )^{-1}[u|_{\Gamma_1^+}-u^{in}|_{\Gamma_1^+}]](x), \quad  \forall x\in D_0'.
	\end{equation}
	It follows from \eqref{modedecototalu} that
	\[
	\begin{aligned}
	 (\mathcal{S}_{\Gamma_1^+}^{k})^{-1}[u|_{\Gamma_1^+}-u^{in}|_{\Gamma_1^+}] &= -4\pi r_1^+ u^{in}(0) \sum_{i=1}^N \frac{1}{\lambda_i}\frac{ \({ \bm{a}^{(1)}_{i}}\)^2}{\frac{\omega^2}{\omega_{M,i}^2}-1+\i \gamma_i } (1+O(\delta^{\frac{1}{2}}))  \mathcal{S}_{\Gamma_1^+}^{-1}  [\chi_{\Gamma_1^+}] +O(1)\\
	& = 4\pi u^{in}(0) \sum_{i=1}^N \frac{1}{\lambda_i}\frac{ \({ \bm{a}^{(1)}_{i}}\)^2}{\frac{\omega^2}{\omega_{M,i}^2}-1+\i \gamma_i } (1+O(\delta^{\frac{1}{2}}))   +O(1).
	\end{aligned}
	\]
	This, together with \eqref{usc} and  for  $|x|\to\infty$ and  $\delta\to 0$,
	\[
	\mathcal{S}_{\Gamma_1^+}^{k} [\psi_1^+](x) = \(\int_{\Gamma_1^+}\psi_1^+ ~\d \sigma\)\(1+O(\delta^{\frac{1}{2}})+O(|x|^{-1})\)G_k(x),
	\]
concludes the proof.
\end{proof}

\section{Numerical computations}\label{sec5}

In this section, we conduct numerical simulations to corroborate our theoretical findings in the previous sections. We first analyze the mode splitting in multi-layer concentric balls. We now have both an asymptotic  approach and a numerical method for computing the eigenfrequencies of MLHC resonators. It is valuable to compare the virtues of these two methods.
Moreover, it is important to understand the acoustic pressure distribution associated with each eigenfrequency.

\subsection{Mode splitting}\label{subsecMS}
In this subsection, we shall compute  the eigenfrequencies.
To validate  the eigenfrequencies formulas in Theorem \ref{thmeigenfreqc}, we first numerically compute the characteristic value $\omega_N^{(c)}$ of $\bm{A}(\omega, \delta):=\bm{A}_{(0)}(\omega, \delta)$ in \eqref{4Nby4N} based on the spherical wave expansions.  Comparing $\omega_N^{(c)}$ and the exact formulas \eqref{equeigenfreqc} for the eigenfrequencies, denoted by $\omega_N^{(e)}$, based on the eigenvalues of the generalized capacitance matrix $\bm{V}^{-1}\mathcal{C}$ in \eqref{GEP}.

In what follows, we set $\rho_\rmr = \kappa_\rmr = 1$, $\rho = \kappa = 1/\delta$, and $\delta =1/6000$.
Setting $ f(\omega) = \det(\bm{A}(\omega, \delta))$, we have that calculating $\omega_N^{(c)}$ is equivalent to determining the following complex root-finding problem
\begin{equation}\label{CRFP}
\omega_N^{(c)} = \min\limits_{\omega \in \mathbb{C}}\{\omega| \ f(\omega) = 0\},
\end{equation}
which can be calculated by using Muller's method \cite{AK_book2018}.
We consider the radius of layers are equidistance.  For $N$-layer structure, set
\begin{equation}\label{str01}
r_i^+=(N-i+1)\;\text{ and }\;r_i^-=(N-i+0.5), \; \quad i=1, 2, \ldots N.
\end{equation}
The comparison between the two methods for computing the eigenfrequencies with positive real parts in the 24-layer nested resonators is shown in Figure \ref{48layerd1_6000}. The plot presents  both the numerical results obtained using the spherical wave expansions outlined in Section \ref{FCR} and the asymptotic results derived from the eigenvalues of the generalized capacitance matrix specified in Theorem \ref{thmeigenfreqc}. The results from both methods are in excellent agreement.
The primary distinction between the two methods lies in their computational cost. In this example, when run on a laptop, the spherical wave expansion  method required 18 seconds to compute, whereas the capacitance matrix approximation took only 0.03 seconds. This significant difference arises because the numerical method involves root-finding procedures, whereas the asymptotic method bypasses this step, leading to a reduction in computational cost by several orders of magnitude.

\begin{figure}[h]
	\centering
	\includegraphics[scale=0.4]{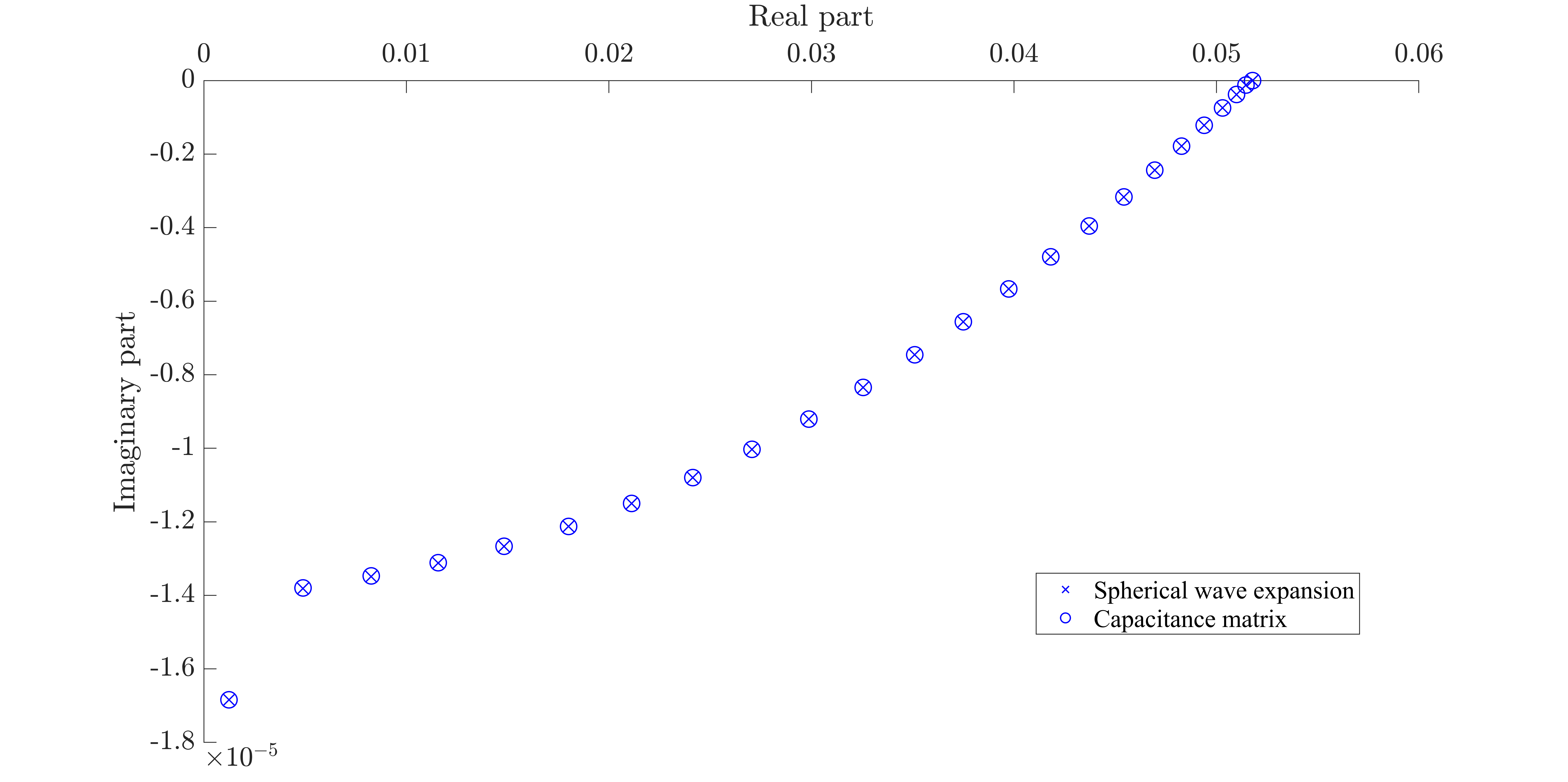}
	\caption{ The subwavelength resonant frequencies, plotted in the complex plane, of the  $24$-layer nested resonators designed by \eqref{str01} with $\delta = 1/6000$. We compare the values computed using the spherical wave expansion  and the values computed using the capacitance matrix. The computations using the spherical wave expansion took 18 seconds while the approximations from the capacitance matrix took just 0.03 seconds, on the same computer.}\label{48layerd1_6000}
\end{figure}

\subsection{Resonant modes and field concentration}

In this subsection, we try to understand  the distributions of the acoustic pressure $u$ to equation \eqref{main_equation} under the excitation of the plane wave $u^{in} = \e^{\i\frac{\omega_{in}}{v} x\cdot d}$.

To facilitate visualization of the results, we focus on four-layer nested resonators. The acoustic pressure distributions $u_{(j)}$ for the four-layer nested resonators, as designed by \eqref{str01}, are shown in Figure \ref{8_layer_d_05_6000}. In this case, the direction of incidence is $d = (1,0,0)$, and the impinging frequency is $\omega_{in} = \omega_j^+ $, for $j = 1, 2, 3, 4 $). The eigenfrequencies $\omega_j^+$ are listed in ascending order of their real parts. The corresponding acoustic pressure distributions inherit the symmetry of the nested resonators and exhibit progressively more oscillatory patterns.

It is evident from the lower plot of Figure \ref{8_layer_d_05_6000} that the acoustic pressure remains approximately constant within each resonator, which supports the modal decomposition derived in Theorem \ref{thm41}. Additionally, it is noteworthy that $ u_{(1)} $ maintains the same sign on each resonator $ D_j $, $u_{(2)}$ changes sign only between $D_1$ and $D_2$, $u_{(3)}$ changes sign between $D_i$ and $D_{i+1}$ for $i = 1, 3 $, and $u_{(4)}$ changes sign between $D_i$ and $D_{i+1}$ for $i = 1, 2, 3$.
This behavior may be associated with the phenomenon of field concentration, where the degree of concentration is characterized by the blow-up rate of the underlying gradient field as the distance between two resonators approaches zero. This phenomenon is a central topic in the theory of composite materials. For studies on gradient blow-up in nearly touching, separated resonators, we refer to \cite{DKLZ_AAMM2024, ADY_MMS2020, LZ_MMS2023, DFLMMS22}. In the context of resonators with nested geometries, the gradient blow-up phenomenon has been primarily investigated in electrostatic problems \cite{KM_MA2019, AKLLL_JMPA2007, LXSIAM2017}.  The determination of blow-up rates for the acoustic field is beyond the scope of the present study, and we shall investigate along this direction in a forthcoming work.

In Figures \ref{8_layer_d_05_6000peak} and \ref{20_layer_d_05_6000peak}, we show how the $L^2$-norm of the acoustic pressure $u$ to equation \eqref{main_equation} varies as a function of the impinging frequency $\omega_{in}$ for the four-layer nested resonators and the ten-layer nested resonators designed by \eqref{str01}, respectively. It can be observed that the peaks of the norm of the acoustic pressure appear when $\omega_{in}$ is close to the real part of the eigenfrequencies, providing evidence for the validity of Theorem \ref{thm41}.

\begin{figure}[htbp]
	\centering
	\includegraphics[scale=0.73]{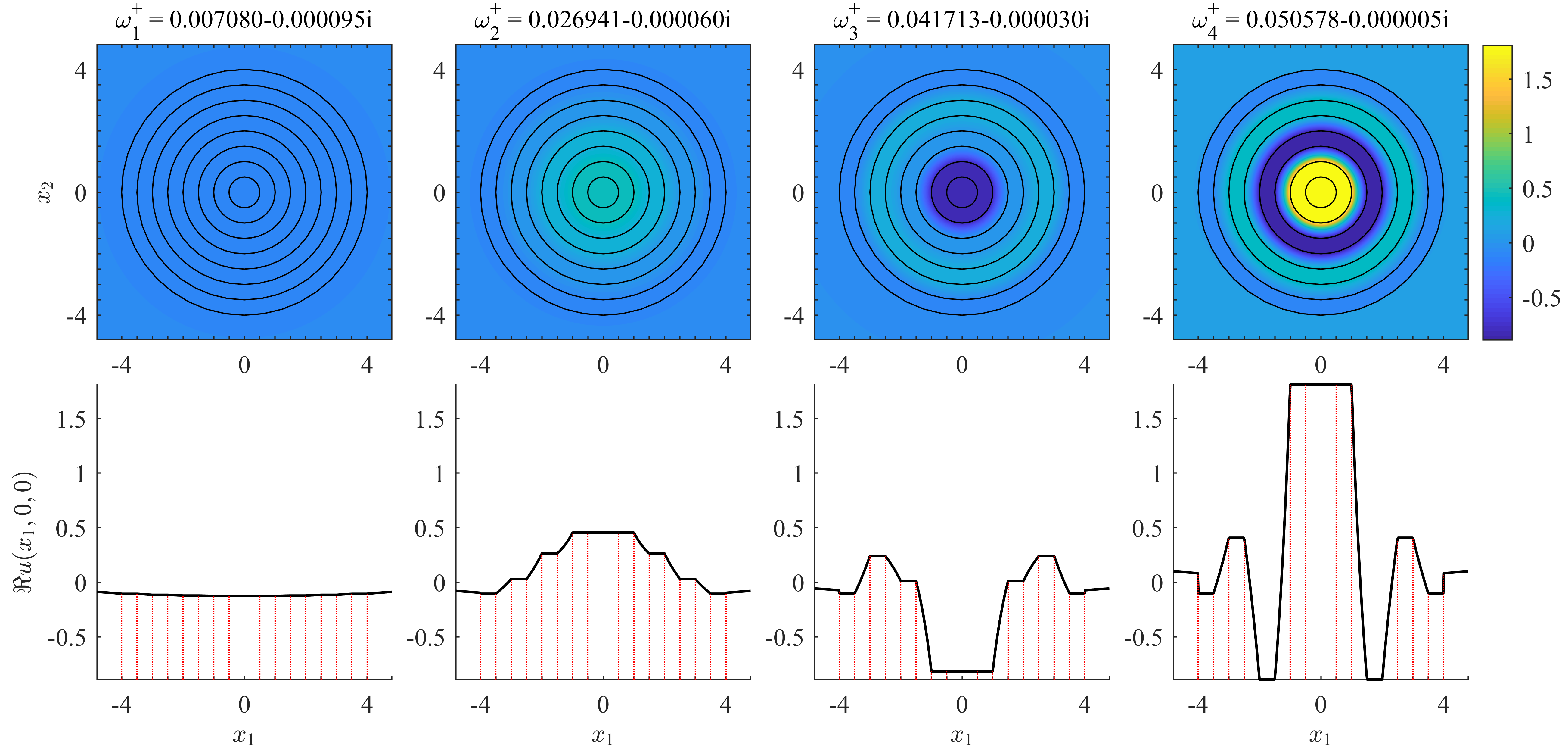}
	\caption{The acoustic pressure distributions $u_{(1)},u_{(2)},u_{(3)},u_{(4)}$ for the four-layer nested resonators designed by \eqref{str01}. Each pair of plots corresponds to one of the four eigenfrequencies. The upper plot displays a contour plot of the function $\Re u_k(x_1, x_2,0)$, with the seven-layer concentric ball designed by \eqref{str01} represented as solid black lines. The lower plot shows the cross section of the upper plot, taken along the line $x_2 = 0$ (passing through the centres of the multi-layered structures). Additionally, red dotted lines represent vertical lines at the coordinates of the radius.}\label{8_layer_d_05_6000}
\end{figure}
\begin{figure}[htbp]
	\centering
	\includegraphics[scale=0.5]{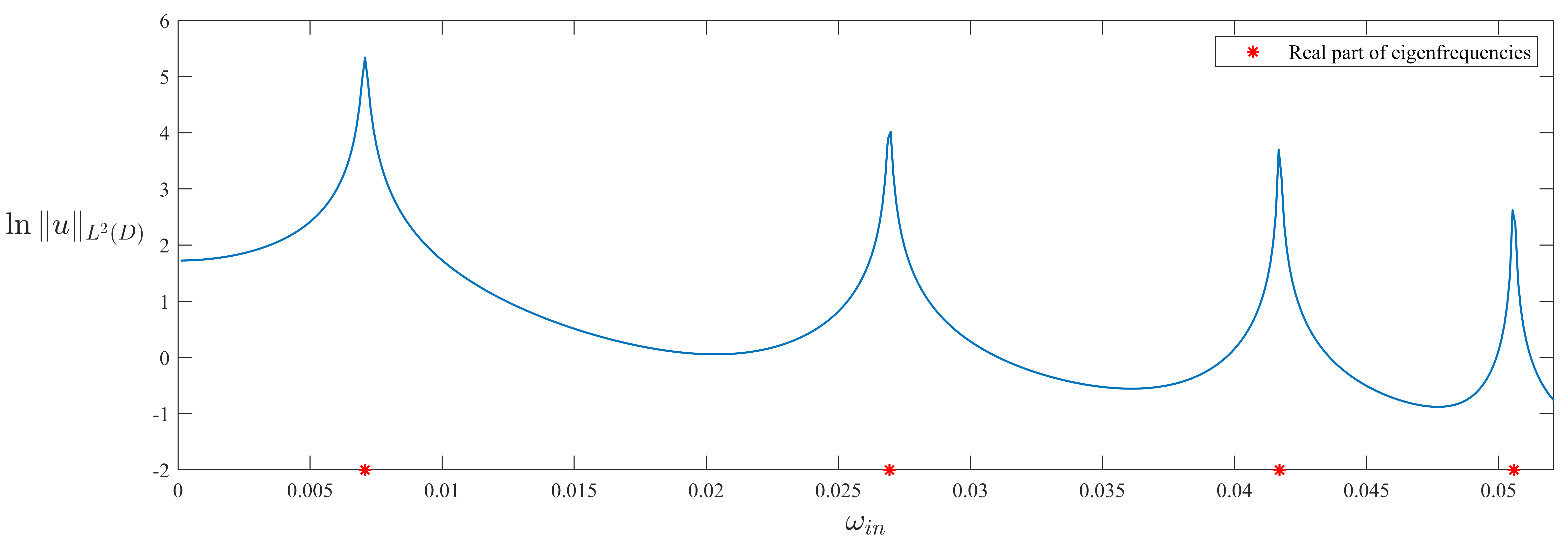}
	\caption{
		Norm of the acoustic pressure $u$ to equation \eqref{main_equation} for the four-layer nested resonators designed by \eqref{str01}.
	}\label{8_layer_d_05_6000peak}
\end{figure}
\begin{figure}[htbp]
	\centering
	\includegraphics[scale=0.5]{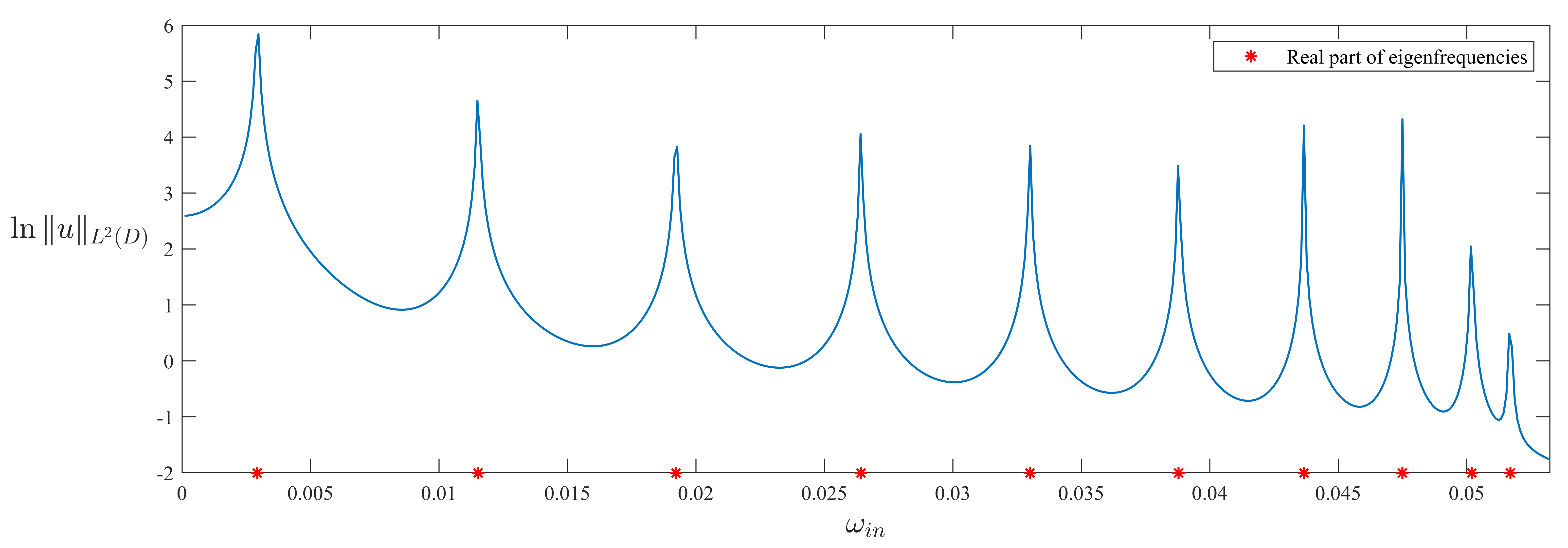}
	\caption{Norm of the acoustic pressure $u$ to equation \eqref{main_equation} for the ten-layer nested resonators designed by \eqref{str01}.}\label{20_layer_d_05_6000peak}
\end{figure}

\section{Concluding remarks}\label{sec6}
In this paper, we conducted  a comprehensive mathematical analysis of subwavelength resonances in multi-layered high-contrast structures. We established a relationship between the corresponding subwavelength resonant frequencies and the eigenvalues of a tridiagonal capacitance matrix and provided the  modal decomposition and  monopole approximation of multi-layered high-contrast structures. The discovery of this tridiagonal structure is a noteworthy result, which opens the door for further exploration. In particular, the explicit form of the capacitance matrix facilitates the study of resonance properties in phononic crystals composed of multiple multi-layered structures. This may be achieved through the application of established spectral theory for Toeplitz matrices, particularly within the context of one-dimensional topological metamaterials \cite{ABSCDE_ARMA2024, FCA_SIAP2023, ABL_arXiv1, ABBLT_arXiv, ABDHL_SIAP2024, ABL_arXiv2}.  These new developments will be reported in our forthcoming works.

\section*{Acknowledgment}
This work was partially supported by NSFC-RGC Joint Research Grant No. 12161160314 and the Natural Science Foundation Innovation Research Team Project of Guangxi (Grant No. 2025GXNSFGA069001).


\end{document}